\documentclass{amsart}
\usepackage{amssymb,latexsym}
\usepackage{amscd,amsthm}
\usepackage{mathrsfs}

\usepackage[all]{xy}
\usepackage{tikz}

\usepackage{tikz-cd}

\newtheorem{theorem}{Theorem}[section]

\newtheorem{lemma}[theorem]{Lemma}
\newtheorem{setup}[theorem]{Setup}
\newtheorem{proposition}[theorem]{Proposition}
\newtheorem{corollary}[theorem]{Corollary}

\theoremstyle{definition}
\newtheorem{definition}[theorem]{Definition}

\newtheorem*{remark}{Remark}

\newtheorem*{example}{Example}

\DeclareMathOperator{\Ext}{Ext}
\DeclareMathOperator{\Hom}{Hom}
\DeclareMathOperator{\Tor}{Tor}

\DeclareMathOperator{\im}{Im}

\newcommand{\cat}[1]{\mathcal{#1}}           

\newcommand{\tensor}{\otimes}

\newcommand{\class}[1]{\mathcal{#1}}   

\newcommand{\Z}{\mathbb{Z}}
\newcommand{\Q}{\mathbb{Q/Z}}
\newcommand{\mathcolon}{\colon\,} 

\newcommand{\ch}{\textnormal{Ch}(R)}
\newcommand{\cha}[1]{\textnormal{Ch}(\mathcal{#1})}
\newcommand{\rmod}{R\text{-Mod}}

\newcommand{\tilclass}[1]{\widetilde{\class{#1}}}
\newcommand{\dgclass}[1]{dg\widetilde{\class{#1}}}

\newcommand{\dwclass}[1]{dw\widetilde{\class{#1}}}

\newcommand{\rightperp}[1]{#1^{\perp}}
\newcommand{\leftperp}[1]{{}^\perp #1}

\newcommand{\homcomplex}{\mathit{Hom}}

\newcommand{\barAF}{_{A_{\! X}\! }{\widetilde{\mathcal F}}}
\newcommand{\barAFac}{_{A_{\! X}\!\textrm{-ac} }{\widetilde{\mathcal F}}}

\def\Qco{\mathfrak{Qco}}

\begin{document}

\title{The projective stable category of a coherent scheme}

\author{Sergio Estrada}
\address{Departamento de Matem\'aticas \\
         Universidad de Murcia \\
         Campus de Espinardo \\
         Murcia, Spain 30100}
\email[Sergio Estrada]{sestrada@um.es}

\author{James Gillespie}
\address{Ramapo College of New Jersey \\
         School of Theoretical and Applied Science \\
         505 Ramapo Valley Road \\
         Mahwah, NJ 07430}
\email[Jim Gillespie]{jgillesp@ramapo.edu}
\urladdr{http://pages.ramapo.edu/~jgillesp/}

\thanks{The authors are grateful to the Institut Henri Poincar\'e for support
through the \emph{Research in Paris} program.}
\date{\today}

\begin{abstract}
We define the projective stable category of a coherent scheme. It is the homotopy category of an abelian model structure on the category of unbounded chain complexes of quasi-coherent sheaves. We study the cofibrant objects of this model structure, which are certain complexes of flat quasi-coherent sheaves satisfying a special acyclicity condition.
\end{abstract}

\maketitle

\section{introduction}
Let $R$ be a ring and $R$-Mod the category of left $R$-modules. The projective stable module category of $R$ was introduced in~\cite{bravo-gillespie-hovey}. The construction provides a triangulated category $\cat{S}_{prj}$ and a product-preserving functor $\gamma \mathcolon R\text{-Mod} \xrightarrow{} \cat{S}_{prj}$ taking short exact sequences to exact triangles, and kills all injective and projective modules (but typically will kill more than just these modules). The motivation for this paper is to extend this construction to schemes, that is, to replace $R$ with a scheme $X$ and introduce the projective stable (quasi-coherent sheaf) category of a scheme $X$. Although we don't yet understand the situation in full generality, this paper does make significant progress towards this goal.

Let us back up and explain the projective stable module category of a ring $R$. The idea is based on a familiar concept in Gorenstein homological algebra, that of a totally acyclic complex of projectives.
These are exact complexes $P$ of projective $R$-modules such that $\Hom_R(P,Q)$ is also exact for all projective $R$-modules $Q$. Such complexes historically arose in group cohomology theory, since the Tate cohomology groups are defined using totally acyclic complexes of projectives. The essentials of the theory hold for  Noetherian rings $R$, assuming a hypothesis on the class of flat modules (which is always satisfied in the case that $R$ has a dualizing complex).
A key idea from~\cite{bravo-gillespie-hovey} is that if we drop this last hypothesis then the theory still works for a general Noetherian ring, and in fact for \emph{all} rings, at the cost of replacing the totally acyclic complexes of projectives with the stronger notion of firmly acyclic complexes of projectives. For any ring $R$, an exact complex $P$ of projectives is called \emph{firmly acyclic} if $\Hom_R(P,F)$ is also exact for all level $R$-modules $F$. Level modules are defined in Section~\ref{sec-preliminaries}, but they are nothing more than flat modules whenever $R$ is a coherent ring, and in particular when $R$ is a Noetherian ring. The firmly acyclic complexes of projectives appear to have first arisen in the work of Nanqing Ding and coauthors. For example, they appear explicitly in~\cite{ding and mao 08}. The projective stable module category, $\cat{S}_{prj}$, is equivalent to $K_{fir}(Proj)$, the chain homotopy category of all firmly acyclic complexes of projectives.

Now in~\cite{bravo-gillespie-hovey}, $K_{fir}(Proj)$ was constructed as the homotopy category of a cofibrantly generated abelian model stucture on $\ch$, the category of chain complexes of $R$-modules. The cofibrant objects in this model structure are precisely the firmly acyclic complexes of projectives. The first main point made in this paper, formally stated in Corollary~\ref{cor-exact-AC-flat-model},  is that $K_{fir}(Proj)$ is the homotopy category of another cofibrantly generated abelian model structure on $\ch$.  The cofibrant objects in this model structure are exact complexes $F$ of flat modules that are \emph{AC-acyclic} in the sense that $A \tensor_R F$ remains exact for all absolutely clean modules $A$. Absolutely clean modules are also defined in Section~\ref{sec-preliminaries}, but they are nothing more than absolutely pure modules (also called FP-injective modules) whenever $R$ is a coherent ring. We call this model structure the \emph{exact AC-acyclic flat model structure}, and we think of it as a ``flat model'' for $K_{fir}(Proj)$; see Corollary~\ref{cor-flat-model}.

The existence of the exact AC-acyclic model structure indicates that, in the spirit of~\cite{gillespie},~\cite{gillespie-sheaves},~\cite{gillespie-quasi-coherent},~\cite{murfet-thesis},~\cite{neeman-flat}, and~\cite{murfet-salarian}, the category $K_{fir}(Proj)$ might extend to non-affine schemes $X$. The theme in each of these papers is to introduce some triangulated category based on flat sheaves to replace some hole left by the fact that there are not enough projectives. So to proceed, for a scheme $X$ with structure sheaf $\mathscr{R}$, we define an exact chain complex $\mathscr{F}_{\bullet}$ of flat quasi-coherent sheaves to be \emph{AC-acyclic} if $\mathscr{F}_{\bullet}(U)$ is an AC-acyclic complex for every open affine $U$ of $X$; that is, if $A_U \otimes_{\mathscr{R}(U)}\mathscr{F}_{\bullet}(U)$ is exact for every absolutely clean $\mathscr{R}(U)$-module $A_U$. Assuming that $X$ is quasi-compact and semi-separated, Corollary~\ref{cor-exact-AC-flat-model-sheaf} shows that the exact AC-acyclic flat model structure does exist on complexes of quasi-coherent sheaves and recovers $K_{fir}(Proj)$ in the affine case $X = \text{Spec}\,R$. Of course one also wants it to agree under their hypotheses, and it does, with Murfet and Salarian's category $D_{tac}(Flat X)$ from~\cite[Definition~4.12]{murfet-salarian}. This was the first paper to consider these topics, focusing instead on extending $K_{tac}(Proj)$, the chain homotopy category of totally acyclic complexes of projectives.

But to be a good notion for complexes of flat sheaves, one wants the notion of AC-acyclicity to be a \emph{Zariski-local} property; that is, a property that can be checked by using any open affine cover of $X$. Unfortunately, we do not know whether or not the notion of AC-acyclic complexes of flats is always a Zariski-local property. However, using recent work in~\cite{chris-estrada-iacob}, we can argue that it is indeed a Zariski-local property whenever the underlying scheme $X$ is locally coherent. This just means that the structure sheaf $\mathscr{R}$ satisfies $\mathscr{R}(U)$ is a coherent ring for each open affine $U \subseteq X$. Assuming $X$ is also semi-separated, the following summarizes the nice properties of AC-acyclic complexes of flats.

\begin{theorem}\label{them-zariski-local}
Let $X$ be a semi-separated and locally coherent scheme with structure sheaf $\mathscr{R}$. Let $\mathscr{F}_{\bullet}$ be a chain complex of flat quasi-coherent sheaves. Then the following are equivalent:
\begin{enumerate}
\item $\mathscr{A} \otimes \mathscr{F}_{\bullet}$ is exact for all locally absolutely pure quasi-coherent sheaves $\mathscr{A}$; that is, all $\mathscr{A}$ such that  $\mathscr{A}(U)$ is an absolutely pure $\mathscr{R}(U)$-module for every open affine $U \subseteq X$, or, just for all $U$ in some open affine cover of $X$.
\item  $A \otimes_{\mathscr{R}(U)} \mathscr{F}_{\bullet}(U)$ is exact for all open affine $U \subseteq X$, or just for all $U$ in some open affine cover of $X$,  and absolutely pure $\mathscr{R}(U)$-modules $A$.
\item  $I \otimes_{\mathscr{R}(U)} \mathscr{F}_{\bullet}(U)$ is exact for all open affine $U \subseteq X$, or just for all $U$ in some open affine cover of $X$, and injective $\mathscr{R}(U)$-modules $I$.
\end{enumerate}
\end{theorem}

We note that the above also says that, in the coherent case, $\mathscr{F}_{\bullet}$ is \emph{$\mathbf{F}$-totally acyclic} in the sense of~\cite{chris-estrada-iacob} if and only if it is both exact and AC-acyclic. In turn, this notion agrees with the original definition given in~\cite{murfet-salarian} assuming $X$ is locally noetherian.

\begin{proof}
See Definition~\ref{def-AC-sheaves}, Theorem~\ref{them-AC-characterization} and Corollary~\ref{cor-zariski-local}.
\end{proof}

In light of the above we therefore see the following theorem and its corollary as the main results of the paper. A note on language: we call a scheme \emph{coherent} if it is both quasi-compact and locally coherent. A quasi-coherent sheaf $\mathscr{F}$ is called \emph{Gorenstein flat} if $\mathscr{F} = Z_0\mathscr{F}_{\bullet}$ for some $\mathbf{F}$-totally acyclic complex $\mathscr{F}_{\bullet}$. Also, assuming $R$ is a coherent ring, a \emph{Ding projective} module refers to a module $M$ such that $M = Z_0P$ for some exact complex of projectives that is firmly acyclic.

\begin{theorem}\label{them-coherent-scheme-model}
Let $X$ be a semi-separated coherent scheme and
let ${}_A\tilclass{F}$ denote the class of all exact AC-acyclic (equivalently, $F$-totally acyclic) complexes of flat quasi-coherent sheaves. Then there is a cofibrantly generated abelian model structure, the \textbf{exact AC-acyclic flat model structure} on $\textnormal{Ch}(\Qco(X))$, for which ${}_A\tilclass{F}$ is the class of cofibrant objets. The trivially cofibrant objects are precisely the categorically flat chain complexes. In other words, the cofibrant (resp. trivially cofibrant) complexes are precisely the exact complexes of flats having all cycles Gorenstein flat (resp. categorically flat). If $X = \textnormal{Spec}\,R$ is an affine scheme, then the homotopy category of this model structure is equivalent to the homotopy category of all Ding projective modules.
\end{theorem}

The last sentence answers a question left open at the end of~\cite{gillespie-flat stable}. Section~\ref{sec-coherent case} is devoted to making this equivalence abundantly clear. It means that, for a semi-separated coherent scheme $X$, the category $\cat{S}_{prj}$ we seek should be defined as the homotopy category of the exact AC-acyclic flat model structure.
For the proof of Theorem~\ref{them-coherent-scheme-model}, see Corollary~\ref{cor-exact-AC-flat-model-sheaf} and the paragraph that follows it.

At the end of the paper we use recent methods from the theory of abelian model categories to easily obtain a localization sequence; see Definition~\ref{def-localization sequence}, Corollary~\ref{cor-localization} and the remark following it. To briefly describe it, let $\class{M}_2$ be the model structure from the above Theorem~\ref{them-coherent-scheme-model}. We also have an abelian model structure,  $\class{M}_1$, whose homotopy category is equivalent to $\class{D}(Flat X)$, the derived category of flat sheaves from~\cite{murfet-salarian}. (This category was initially denoted $K_m(Proj X)$ and called the \emph{mock homotopy category of projectives} in~\cite{murfet-thesis}.)
Using ideas from~\cite{becker, gillespie-mock projectives}, we obtain a third model structure $\class{M}_2\backslash \class{M}_1$, called the \emph{left localization of $\class{M}_1$ by $\class{M}_2$}.  Then we have the following localization sequence recovering~\cite[Prop.~5.1]{murfet-salarian} under their noetherian hypothesis.

\begin{corollary}\label{cor-localization-sequence}
Continuing Theorem~\ref{them-coherent-scheme-model}, there is a localization sequence
\[
\begin{tikzpicture}[node distance=3.5 cm, auto]
\node (A)  {$\textnormal{Ho}(\class{M}_2)$};
\node (B) [right of=A] {$\textnormal{Ho}(\class{M}_1)$};
\node (C) [right of=B] {$\textnormal{Ho}(\class{M}_2\backslash\class{M}_1)$};
%
%
\draw[->] (A.10) to node {$L(\textnormal{Id})$} (B.170);
\draw[<-] (A.350) to node [swap] {$R(\textnormal{Id})$} (B.190);
\draw[->] (B.10) to node {$L(\textnormal{Id})$} (C.173);
\draw[<-] (B.350) to node [swap] {$R(\textnormal{Id})$} (C.187);
\end{tikzpicture}
\]
where $L(\textnormal{Id})$ and $R(\textnormal{Id})$ are left and right derived identity functors on $\textnormal{Ch}(\Qco(X))$ and $\class{M}_2 \backslash \class{M}_1$ is the \emph{left localization of $\class{M}_1$ by $\class{M}_2$}.
\end{corollary}

After a preliminary Section~\ref{sec-preliminaries}, Sections~\ref{sec-acyclicity-modules} and~\ref{sec-AC-acyclic flat model on ch} are devoted to a clear understanding of the affine case. We work over a not necessarily commutative ring with identity. In~\cite{bravo-gillespie-hovey}, there appears some very general constructions of ``injective'' and ``projective'' model structures on $\ch$, leading up to the injective and projective stable module categories of a ring $R$. Section~\ref{sec-AC-acyclic flat model on ch} is the ``flat'' analog providing models that are Quillen equivalent to the corresponding projective model structures built in~\cite{bravo-gillespie-hovey}. This lays the groundwork for Sections~\ref{sec-sheaves} through~\ref{sec-AC-acyclic flat model on sheaf} which build the analogous model structures on complexes of quasi-coherent sheaves. The coherent assumption makes things interesting even in the case of modules over a ring $R$. This is discussed in detail in Section~\ref{sec-coherent case}.

\section{preliminaries}\label{sec-preliminaries}

Throughout the paper $R$ will denote a ring.   The category of (left) $R$-modules will be denoted by $R$-Mod. Starting in Section~\ref{sec-sheaves} the reader is assumed to have a basic understanding of (quasi-coherent) sheaves. Our principal concern is the construction of some hereditary abelian model structures on (unbounded) chain complexes of modules and sheaves, resulting in stable homotopy categories. Some particulars are recalled in this section.

\subsection{Cotorsion pairs and abelian model structures}\label{subsec-abelian model cats}
Let $\cat{A}$ be an abelian category. By definition, a pair of classes $(\class{X},\class{Y})$ in $\cat{A}$ is called a \emph{cotorsion pair} if $\class{Y} = \rightperp{\class{X}}$ and $\class{X} = \leftperp{\class{Y}}$. Here, given a class of objects $\class{C}$ in $\cat{A}$, the right orthogonal  $\rightperp{\class{C}}$ is defined to be the class of all objects $X$ such that $\Ext^1_{\cat{A}}(C,X) = 0$ for all $C \in \class{C}$. Similarly, we define the left orthogonal $\leftperp{\class{C}}$. We call $\class{X} \cap \class{Y}$ the \emph{core} of the cotorsion pair, and we call a cotorsion pair \emph{hereditary} if $\Ext^i_{\cat{A}}(X,Y) = 0$ for all $X \in \class{X}$, $Y \in \class{Y}$, and $i \geq 1$. The cotorsion pair is \emph{complete} if it has enough injectives and enough projectives. This means that for each $A \in \cat{A}$ there exist short exact sequences $0 \xrightarrow{} A \xrightarrow{} Y \xrightarrow{} X \xrightarrow{} 0$ and $0 \xrightarrow{} Y' \xrightarrow{} X' \xrightarrow{} A \xrightarrow{} 0$ with $X,X' \in \class{X}$ and $Y,Y' \in \class{Y}$.
Cotorsion pairs are fundamentally connected to the study of precovers and preenvelopes in relative homological algebra; we refer to the standard reference~\cite{enochs-jenda-book}.

Cotorsion pairs are also fundamentally connected to the theory of abelian model categories. The main theorem of~\cite{hovey} showed that an abelian model structure on $\cat{A}$ is equivalent to a triple $(\class{C},\class{W},\class{F})$ of classes of objects in $\cat{A}$ for which $\class{W}$ is thick and $(\class{C} \cap \class{W},\class{F})$ and $(\class{C},\class{W} \cap \class{F})$ are each complete cotorsion pairs. By \emph{thick} we mean that the class $\class{W}$ is closed under retracts and satisfies that whenever two out of three terms in a short exact sequence are in $\class{W}$ then so is the third. In this case, $\class{C}$ is precisely the class of cofibrant objects of the model structure, $\class{F}$ are precisely the fibrant objects, and $\class{W}$ is the class of trivial objects. We hence denote an abelian model structure by a triple $\class{M} = (\class{C},\class{W},\class{F})$ and sometimes we call it  a \emph{Hovey triple}. We say that $\class{M}$ is \emph{hereditary} if both of the associated cotorsion pairs are hereditary. Finally, by the \emph{core} of an abelian model structure $\class{M} = (\class{C},\class{W},\class{F})$ we mean the class $\class{C} \cap \class{W} \cap \class{F}$.

A recent result appearing in~\cite{gillespie-hovey triples} will prove fundamental to this paper. It says that whenever $(\tilclass{C},\class{F})$ and $(\class{C},\tilclass{F})$ are complete hereditary cotorsion pairs with equal cores and $\tilclass{F} \subseteq \class{F}$, then there is a unique thick class $\class{W}$ yielding a Hovey triple $\class{M} = (\class{C},\class{W},\class{F})$ with $\class{C} \cap \class{W} = \tilclass{C}$ and $\class{W} \cap \class{F} = \tilclass{F}$. Besides~\cite{hovey} we will refer to~\cite{hovey-model-categories} for any other basics from the theory of model categories.

Now suppose that our category $\cat{A}$ has enough projectives. Then by a \emph{projective model structure} we mean a Hovey triple $\class{M}  = (\class{C},  \class{W}, \class{A})$ where $\class{A}$ now denotes the class of all objects in the category. It is easy to see that such a model structure is equivalent to a cotorsion pair $(\class{C},\class{W})$ with $\class{W}$ thick and $\class{C} \cap \class{W}$ coinciding with the class of projectives. We call such a cotorsion pair a \emph{projective cotorsion pair}. On the other hand, we also have \emph{injective cotorsion pairs} and \emph{injective model structures} on categories with enough injectives. We will refer to~\cite{becker, gillespie-recollement} on these notions.

\subsection{Chain complexes on abelian categories}\label{subsec-complexes}
Again, let $\cat{A}$ be an abelian category. We denote the corresponding category of chain complexes by $\cha{A}$. In the case $\cat{A} = \rmod$, we denote it by $\ch$. Our convention is that the differentials of our chain complexes lower degree, so $\cdots
\xrightarrow{} X_{n+1} \xrightarrow{d_{n+1}} X_{n} \xrightarrow{d_n}
X_{n-1} \xrightarrow{} \cdots$ is a chain complex. We also have the chain homotopy category of $\cat{A}$, denoted $K(\cat{A})$. Its objects are also chain complexes but its morphisms are chain homotopy classes of chain maps.
Given a chain complex $X$, the
\emph{$n^{\text{th}}$ suspension of $X$}, denoted $\Sigma^n X$, is the complex given by
$(\Sigma^n X)_{k} = X_{k-n}$ and $(d_{\Sigma^n X})_{k} = (-1)^nd_{k-n}$.
For a given object $A \in \cat{A}$, we denote the \emph{$n$-disk on $A$} by $D^n(A)$. This is the complex consisting only of $A \xrightarrow{1_A} A$ concentrated in degrees $n$ and $n-1$, and 0 elsewhere. We denote the \emph{$n$-sphere on $A$} by $S^n(A)$, and this is the complex consisting only of $A$ in degree $n$ and 0 elsewhere.

\subsection{Localization sequences} We  briefly recall the definition of a localization sequence from~\cite{krause-stable derived cat of a Noetherian scheme}.

\begin{definition}\label{def-localization sequence}
Let $\class{T}' \xrightarrow{F} \class{T} \xrightarrow{G} \class{T}''$ be a sequence of exact functors between triangulated categories. We say it is a \emph{localization sequence} when there exists right adjoints $F_{\rho}$ and $G_{\rho}$ giving a diagram of functors as below with the listed properties.
$$\begin{tikzcd}
\class{T}'
\rar[to-,
to path={
([yshift=0.5ex]\tikztotarget.west) --
([yshift=0.5ex]\tikztostart.east) \tikztonodes}][swap]{F}
\rar[to-,
to path={
([yshift=-0.5ex]\tikztostart.east) --
([yshift=-0.5ex]\tikztotarget.west) \tikztonodes}][swap]{F_{\rho}}
& \class{T}
\rar[to-,
to path={
([yshift=0.5ex]\tikztotarget.west) --
([yshift=0.5ex]\tikztostart.east) \tikztonodes}][swap]{G}
\rar[to-,
to path={
([yshift=-0.5ex]\tikztostart.east) --
([yshift=-0.5ex]\tikztotarget.west) \tikztonodes}][swap]{G_{\rho}}
& \class{T}'' \\
\end{tikzcd}$$
\begin{enumerate}
\item The right adjoint $F_{\rho}$ of $F$ satisfies $F_{\rho} \circ F \cong \text{id}_{\class{T}'}$.
\item The right adjoint $G_{\rho}$ of $G$ satisfies $G \circ G_{\rho} \cong \text{id}_{\class{T}''}$.
\item For any object $X \in \class{T}$, we have $GX = 0$ iff $X \cong FX'$ for some $X' \in \class{T}'$.
\end{enumerate}
\end{definition}

\subsection{Absolutely clean and level modules}\label{subsec-absclean-level}

Now let $R$ be a ring.
We will often refer to~\cite{bravo-gillespie-hovey} for the theory of absolutely clean and level modules. We think of absolutely clean modules as a class of modules possessing the same properties as injective modules over Noetherian rings, and we think of level modules as a class of modules with the same properties as flat modules over coherent rings.

Briefly, note that an $R$-module $I$ is injective if and only if $\Ext^1_R(N,I) = 0$ for all finitely generated modules $N$. Over Noetherian rings, $N$ is finitely generated if and only if it is of type $FP_{\infty}$; that is, if $N$ has a projective resolution by finitely generated projective modules. We call $A$ \emph{absolutely clean} (or \emph{$FP_{\infty}$-injective}) if $\Ext^1_R(N,A) = 0$ for all modules $N$ of type $FP_{\infty}$.  For coherent rings,  absolutely clean modules coincide with the absolutely pure modules (also called FP-injective modules). The $FP_{\infty}$-modules have been studied by Livia Hummel in~\cite{miller-livia}.

On the other hand, a module $L$ is called \emph{level} if $\Tor_1^R(N,L) = 0$   for all (right) $R$-modules $N$ of type $FP_{\infty}$. For coherent rings, we have that the level modules coincide with the flat modules. A fundamental result is the character module duality: A left (resp. right) $R$-module $A$ is absolutely clean if and only if $A^+ = \Hom_{\Z}(A,\Q)$ is a level right (resp. left) $R$-module. Similarly, a left (resp. right) $R$-module $L$ is level if and only if $L^+ = \Hom_{\Z}(L,\Q)$ is an absolutely clean right (resp. left) $R$-module.

\subsection{Modified tensor product and Tor functors}\label{subsec-modified tensor and Tor}
Again $R$ is a ring. We denote by $X \overline{\otimes} Y$, the modified tensor product of chain complexes from~\cite{enochs-garcia-rozas} and~\cite{garcia-rozas}. This is the correct tensor product for characterizing flatness and purity in $\ch$. That is, a complex $F$ is a direct limit of finitely generated projective complexes if and only if $F \overline{\otimes} -$ is an exact functor. And, a short exact sequence $\class{E}$ of chain complexes is \emph{pure} in the sense of Definition~\ref{def-purity} if and only if $X \overline{\otimes} \class{E}$ remains exact for all complexes $X$ (of right $R$-modules). $\overline{\otimes}$ is defined in terms of the usual tensor product $\otimes$ of chain complexes as follows. Given a complex $X$ of right $R$-modules and a complex $Y$ of left $R$-modules, we define $X \overline{\otimes} Y$ to be the complex whose $n$-th entry is $(X \otimes Y)_n / B_n(X \otimes Y)$ with boundary map  $(X \otimes Y)_n / B_n(X \otimes Y) \rightarrow (X \otimes Y)_{n-1} / B_{n-1}(X \otimes Y)$ given by
\[
\overline{x \otimes y} \mapsto \overline{dx \otimes y}.
\]
This defines a complex and we get a bifunctor $ - \overline{\otimes} - $ which is right exact in each variable.  We refer the reader to~\cite{garcia-rozas} for more details.

Note that since $M \otimes -$ is right exact, given any complex $X$ we have, for all $n$, a right exact sequence
\[
  M \otimes_R X_{n+1} \to M \otimes_R X_n \to M \otimes_R \left( {X_n}/{B_nX} \right) \to 0.
\]
Therefore ${M \otimes_R X_n} / B_n(M \otimes_R X) \cong M \otimes_R \left( {X_n}/{B_nX} \right) $.

\section{Acyclicity for complexes of flat modules}\label{sec-acyclicity-modules}

Let $R$ be a ring. In this section we introduce and study the notion of $A$-acyclic and (exact) AC-acyclic complexes of flat modules. The main results, which we apply in the next section, are Theorem~\ref{them-deconstructibility of tensor acyclic complexes} and Corollary~\ref{cor-AC acyclic} which show that these complexes are part of complete hereditary cotorsion pairs in $\ch$.

We use purity methods to obtain the results. Purity in the category of chain complexes is defined the same way it is in any other locally finitely presented category.

\begin{definition}\label{def-purity}
A short exact sequence $\class{E} : 0 \xrightarrow{} P \xrightarrow{} X \xrightarrow{} Y \xrightarrow{} 0$ of chain complexes is called \textbf{pure} if $\Hom_{\ch}(F,\class{E})$ remains exact for any finitely presented chain complex $F$. In the same way we say $P \subseteq X$ is a \textbf{pure subcomplex} and $X/P \cong Y$ is a \textbf{pure quotient}.
\end{definition}

By definition, a chain complex $F$ is said to be \emph{finitely presented} if the functor $\Hom_{\ch}(F,-)$ preserves direct limits. But a simpler characterization is that these are precisely the complexes $F$ which are bounded above and below and with each $F_n$ a finitely presented $R$-module~\cite{garcia-rozas}. Using this fact, we prove the following useful lemma concerning properties of pure exact sequences of complexes.

\begin{lemma}\label{lemma-props of pure subcomplexes}
Let $\class{E} : 0 \xrightarrow{} P \xrightarrow{} X \xrightarrow{} Y \xrightarrow{} 0$ be a pure exact sequence of chain complexes. Then the following hold.
\begin{enumerate}
\item Each $0 \xrightarrow{} P_n \xrightarrow{} X_n \xrightarrow{} Y_n \xrightarrow{} 0$ is pure in $R$-Mod.
\item Each $0 \xrightarrow{} Z_nP \xrightarrow{} Z_nX \xrightarrow{} Z_nY \xrightarrow{} 0$ is pure in $R$-Mod.
\item Each $0 \xrightarrow{} P_n/B_nP \xrightarrow{} X_n/B_nX \xrightarrow{} Y_n/B_nY \xrightarrow{} 0$ is pure in $R$-Mod.
\item If $X$ is exact, then so is $P$ and $Y$. That is, the class of exact complexes is closed under pure subcomplexes and pure quotients.
\end{enumerate}
\end{lemma}

\begin{proof}
If $M$ is a finitely presented module, then $S^n(M)$ and $D^n(M)$ are finitely presented complexes. So both of $\Hom_{\ch}(S^n(M),\class{E})$ and $\Hom_{\ch}(D^n(M),\class{E})$ remain exact. Using the standard isomorphisms $$\Hom_{\ch}(S^n(M),X) \cong \Hom_{R}(M,Z_nX)$$ and $\Hom_{\ch}(D^n(M),X) \cong \Hom_R(M,X_n)$ we conclude that (1) and (2) hold. Note that in particular, $0 \xrightarrow{} Z_nP \xrightarrow{} Z_nX \xrightarrow{} Z_nY \xrightarrow{} 0$ must be a short exact sequence since $$\Hom_{\ch}(S^n(R),X) \cong \Hom_{R}(R,Z_nX) \cong Z_nX.$$

To prove (3), we use the modified tensor product of Enochs and Garc\'\i{}a-Rozas from Section~\ref{subsec-modified tensor and Tor}. If $M$ is a finitely presented (right) module, then again $S^0(M)$ is a finitely presented complex. So $0 \xrightarrow{} S^0(M) \overline{\otimes} P \xrightarrow{} S^0(M) \overline{\otimes} X \xrightarrow{} S^0(M) \overline{\otimes} Y \xrightarrow{} 0$ remains exact by~\cite[Theorem~5.1.3]{garcia-rozas}. Note that by the definition of $\overline{\otimes}$, this is just the following short exact sequence of complexes.
$$\begin{CD}
 @. \vdots @. \vdots @. \vdots  \\
@. @VV0V @VV0V @VV0V @. \\
 0 @>>> \frac{M \otimes_R P_{n+1}}{B_{n+1}(M \otimes_R P)} @>>> \frac{M \otimes_R X_{n+1}}{B_{n+1}(M \otimes_R X)} @>>> \frac{M \otimes_R Y_{n+1}}{B_{n+1}(M \otimes_R Y)}  @>>> 0 \\
@. @VV0V @VV0V @VV0V @. \\
 0 @>>> \frac{M \otimes_R P_{n}}{B_{n}(M \otimes_R P)} @>>> \frac{M \otimes_R X_{n}}{B_{n}(M \otimes_R X)} @>>> \frac{M \otimes_R Y_{n}}{B_{n}(M \otimes_R Y)}  @>>> 0 \\
@. @VV0V @VV0V @VV0V @. \\
 0 @>>> \frac{M \otimes_R P_{n-1}}{B_{n-1}(M \otimes_R P)} @>>> \frac{M \otimes_R X_{n-1}}{B_{n-1}(M \otimes_R X)} @>>> \frac{M \otimes_R Y_{n-1}}{B_{n-1}(M \otimes_R Y)}  @>>> 0 \\
@. @VV0V @VV0V @VV0V @. \\
 @. \vdots @. \vdots @. \vdots  \\
\end{CD}$$
But note that for any complex $X$, we have $\frac{M \otimes_R X_{n}}{B_{n}(M \otimes_R X)} \cong M \otimes_R (X_n/B_nX)$. (Indeed $M \otimes_R -$ preserves right exact sequences. So for each $n$, we have the right exact sequence $M \otimes_R X_{n+1} \xrightarrow{} M \otimes_R X_n \xrightarrow{} M \otimes_R (X_n/B_nX) \xrightarrow{} 0$.) So through this isomorphism, we see that for each $n$, we have a short exact sequence
$$0 \xrightarrow{} M \otimes_R (P_n/B_nP) \xrightarrow{} M \otimes_R (X_n/B_nX) \xrightarrow{} M \otimes_R (Y_n/B_nY) \xrightarrow{} 0.$$
We conclude  that $$0 \xrightarrow{} P_n/B_nP \xrightarrow{} X_n/B_nX \xrightarrow{} Y_n/B_nY \xrightarrow{} 0$$ is a pure exact sequence in $R$-Mod.

To prove (4), we now apply the snake lemma to
$$\begin{CD}
 0 @>>> Z_nP @>>> Z_nX @>>> Z_nY  @>>> 0 \\
@. @VVV @VVV @VVV @. \\
 0 @>>> P_n @>>> X_n @>>> Y_n  @>>> 0 \\
\end{CD}$$
to conclude we have a short exact sequence
$0 \xrightarrow{} B_{n-1}P \xrightarrow{} B_{n-1}X \xrightarrow{} B_{n-1}Y \xrightarrow{} 0$ for all $n$.
We then turn around and again apply the snake lemma to
$$\begin{CD}
 0 @>>> B_nP @>>> B_nX @>>> B_nY  @>>> 0 \\
@. @VVV @VVV @VVV @. \\
 0 @>>> Z_nP @>>> Z_nX @>>> Z_nY  @>>> 0 \\
\end{CD}$$
and use that $B_nX = Z_nX$ to conclude that $B_nP = Z_nP$ and $B_nY = Z_nY$.

\end{proof}

The next proposition will require the following lemma whose proof can be found in~\cite[Lemma~5.2.1]{garcia-rozas} or~\cite[Lemma~4.6]{gillespie}. For a chain complex $X$, we define its cardinality to be
$|\coprod_{n \in \Z} X_n|$.

\begin{lemma}\label{lemma-pure cardinality}
Let $\kappa$ be some regular cardinal with $\kappa > |R|$.
Say $X \in \ch$ and $S \subseteq X$ has $|S| \leq \kappa$. Then there exists a pure $P \subseteq X$ with
$S \subseteq P$ and $|P| \leq \kappa$.
\end{lemma}

\begin{remark}\label{remark-homcomplex}
We note that~\cite[Lemma~5.2.1]{garcia-rozas} and~\cite[Lemma~4.6]{gillespie} give several other characterizations of pure exact sequences of complexes, but none of them are stated exactly the same as our definition above. However, they are equivalent. In particular, one of their characterizations of purity is that the enriched Hom-complex functor $\overline{\homcomplex}(F,-)$ remains an exact sequence (of complexes) for any finitely presented complex $F$. However, for chain complexes $X$,$Y$, the definition of $\overline{\homcomplex}(X,Y)$ turns out to just be $\Hom_{\ch}(X,\Sigma^{-n}Y)$ in degree $n$. So indeed, $\overline{\homcomplex}(F,-)$ preserves a short exact sequence for all finitely presented $F$ if and only if $\Hom_{\ch}(F,-)$ preserves the same short exact sequence for all finitely presented $F$.
\end{remark}

\begin{proposition}\label{prop-transfinite}
Suppose $\mathcal{A}$ is a class of chain complexes that is closed under taking pure subcomplexes and pure quotients. Then there is a regular cardinal $\kappa$ such that every chain complex in $\mathcal{A}$ is a transfinite extension of complexes in $\mathcal{A}$ with cardinality bounded by $\kappa$, meaning $\leq \kappa$.
\end{proposition}

\begin{proof}
As in Lemma~\ref{lemma-pure cardinality}, we let $\kappa$ be some regular cardinal with $\kappa > |R|$. Let $A \in \class{A}$. If $|A| \leq \kappa$ there is nothing to prove. So assume $|A| > \kappa$. We will use transfinite induction to find a strictly increasing continuous chain $0 \neq A_0 \subsetneq A_1 \subsetneq A_2 \subsetneq \cdots \subsetneq A_{\alpha} \subsetneq \cdots$
of subcomplexes of $A$ with each $A_{\alpha}, A_{\alpha+1}/A_{\alpha} \in \class{A}$ and with $|A_0| , |A_{\alpha+1}/A_{\alpha}| \leq \kappa$. We start by applying Lemma~\ref{lemma-pure cardinality} to find a nonzero pure subcomplex $A_0 \subset A$ with $|A_0| \leq \kappa$. Then $A_0$ and $A/A_0$ are each complexes in $\class{A}$ by assumption. So we again apply Lemma~\ref{lemma-pure cardinality} to $A/A_0$ to obtain a nonzero pure subcomplex $A_1/A_0 \subset A/A_0$ with $|A_1/A_0| \leq \kappa$. So far we have $0 \neq A_0 \subsetneq A_1 \subsetneq A$ and with $A_0,A_1/A_0$ back in $\class{A}$ and with their cardinalities bounded by $\kappa$.

We now pause to point out the important fact that $A_1 \subset A$ is also a pure subcomplex. Indeed, given a finitely presented complex $F$, we need to argue that $\Hom_{\ch}(F,A) \rightarrow \Hom_{\ch}(F,A/A_1)$ is an epimorphism. But after identifying $(A/A_0)/(A_1/A_0) \cong A/A_1$, this map is just the composite $$\Hom_{\ch}(F,A) \xrightarrow{} \Hom_{\ch}(F,A/A_0) \xrightarrow{} \Hom_{\ch}(F,(A/A_0)/(A_1/A_0)),$$ and each of these are epimorphisms because $A_0 \subset A$ is pure and $A_1/A_0 \subset A/A_0$ is pure.

Back to the increasing chain $0 \neq A_0 \subsetneq A_1 \subsetneq A$, we also note that $A/A_1$ is back in $\class{A}$ since $(A/A_0)/(A_1/A_0) \cong A/A_1$ is a pure quotient.
So we may repeat the above procedure to construct a strictly increasing chain $0 \neq A_0 \subsetneq A_1 \subsetneq A_2 \subsetneq \cdots$ where each $A_n$ is a pure subcomplex of $A$ and each $A_{n+1}/A_n \in \class{A}$ has cardinality bounded by $\kappa$. We set $A_{\omega} = \cup_{n<\omega} A_n$. Then $A_{\omega}$ is also a pure subcomplex since pure subcomplexes are easily seen to be closed under direct unions; for example, see~\cite[pp.~3384]{gillespie}. So $A_{\omega}$ and $A/A_{\omega}$ are also each in $\class{A}$ and we may continue to build the continuous chain $$0 \neq A_0 \subsetneq A_1 \subsetneq A_2 \subsetneq \cdots \subsetneq A_{\omega} \subsetneq A_{\omega +1} \cdots$$ Continuing with transfinite induction, and setting $A_{\gamma} = \cup_{\alpha < \gamma} A_{\alpha}$ whenever $\gamma$ is a limit ordinal, this process eventually must terminate and we end up with $A$ expressed as a union of a continuous chain with all the desired properties.
\end{proof}

\subsection{Acyclicity of complexes for a given module $A$}

We will now investigate the following general notion of $A$-acyclicity for a given $R$-module $A$.
\begin{definition}\label{def-A-acyclicity}
Suppose we are given a right $R$-module $A$. Then a chain complex $X$ of flat modules, including perhaps  a complex of projectives, will be called \textbf{$\boldsymbol{A}$-acyclic} if $A \tensor_R X$ is exact.
\end{definition}

Note that we do not necessarily assume $X$ to be exact from the start. But note also that $X$ is $(A \oplus R)$-acyclic if and only if $X$ is exact and $A$-acyclic.
The next lemma is the key to the simple proof of the theorem that follows it.

\begin{lemma}[$A$-acyclicity lemma]\label{lemma-A-acyclic}
 For any chain complex $$X = \cdots \to X_{n+1} \xrightarrow{d_{n+1}} X_n \xrightarrow{d_n} X_{n-1} \xrightarrow{d_{n-1}} \cdots $$ there is a canonical factorization of the differentials $d_n$ as  $$d_n = X_n \xrightarrow{\pi_n} \frac{X_n}{B_nX} \xrightarrow{\bar{d}_n} X_{n-1}$$ satisfying the following:
\begin{enumerate}
\item $\pi_n$ is projection and $\bar{d}_n$ is defined as $\bar{d}_n(\bar{x}) = d_n(x)$.
\item For each $n$, we have the (right) exact sequence $$\frac{X_{n+1}}{B_{n+1}X} \xrightarrow{\bar{d}_{n+1}} X_n \xrightarrow{\pi_n} \frac{X_n}{B_nX} \xrightarrow{} 0.$$
\item Let $A$ be a right $R$-module.
Then the complex $A \otimes_R X$ is exact if and only if for each $n$, the induced sequence below is short exact
 $$0 \xrightarrow{} A \otimes_R \frac{X_{n+1}}{B_{n+1}X} \xrightarrow{1_A \otimes \bar{d}_{n+1}} A \otimes_R X_n \xrightarrow{1_A \otimes \pi_n} A \otimes_R \frac{X_n}{B_nX} \xrightarrow{} 0.$$ That is, if and only if applying $A \otimes_R -$ turns each $\frac{X_{n+1}}{B_{n+1}X} \xrightarrow{\bar{d}_{n+1}} X_n$ into a monomorphism.
\end{enumerate}
\end{lemma}

\begin{proof}
Note that the map $\bar{d}_n$ is well-defined since $B_nX \subseteq Z_nX$, and then it is clear that $d_n = \bar{d}_n \circ \pi_n$. The right exactness of the sequence in (2) follows from the fact that $\im{(\bar{d}_{n+1})} = B_nX$.

For (3), applying $A \otimes_R -$ to $\cdots \to X_{n+1} \xrightarrow{d_{n+1}} X_n \xrightarrow{d_n} X_{n-1} \xrightarrow{d_{n-1}} \cdots $, we observe that we still have the factorizations $$1_A \otimes d_n = A \otimes_R X_n \xrightarrow{1 \otimes \pi_n} A \otimes_R \frac{X_n}{B_nX} \xrightarrow{1 \otimes \bar{d}_n} A \otimes_R X_{n-1}$$ satisfying that each
 $$A \otimes_R \frac{X_{n+1}}{B_{n+1}X} \xrightarrow{1_A \otimes \bar{d}_{n+1}} A \otimes_R X_n \xrightarrow{1_A \otimes \pi_n} A \otimes_R \frac{X_n}{B_nX} \xrightarrow{} 0$$ is right exact. Note also that in general we always have $$B_n(A \otimes_R X) = \im{(1_A \otimes d_{n+1})} = \im{(1_A \otimes \bar{d}_{n+1})} = \ker{(1 \otimes \pi_n)}$$ $$ \subseteq \ker{(1 \otimes d_n)} = Z_n(A \otimes_R X).$$ So $A \otimes_R X$ is exact if and only if $\ker{(1 \otimes \pi_n)} = \ker{(1 \otimes d_n)}$. But one can check that this happens if and only if $1 \otimes \bar{d}_n$ is a monomorphism.
\end{proof}

\begin{theorem}\label{them-deconstructibility of tensor acyclic complexes}
Let $R$ be any ring and $A$ be a given right $R$-module. Let ${}_A\tilclass{F}$ denote the class of all $A$-acyclic complexes of flat modules; that is, complexes $F$ that are degreewise flat and such that $A \tensor_R F$ is exact. Then
$({}_A\tilclass{F}, \rightperp{{}_A\tilclass{F}})$ is a complete hereditary cotorsion pair cogenerated by a set. Moreover, every chain complex has a surjective ${}_A\tilclass{F}$-cover.
\end{theorem}

\begin{proof}
We first show that
${}_A\tilclass{F}$ is closed under pure subcomplexes and pure quotients.
So suppose that $F \in {}_A\tilclass{F}$, and assume  $0 \xrightarrow{} P \xrightarrow{} F \xrightarrow{} Y \xrightarrow{} 0$ is a pure exact sequence.  We have the commutative diagram below with (right) exact columns by part~(2) of Lemma~\ref{lemma-A-acyclic} and pure exact rows by parts~(1) and~(3) of Lemma~\ref{lemma-props of pure subcomplexes}.
$$\begin{CD}
 0 @>>> \frac{P_{n+1}}{B_{n+1}P} @>>> \frac{F_{n+1}}{B_{n+1}F} @>>> \frac{Y_{n+1}}{B_{n+1}Y}  @>>> 0 \\
@. @VV\bar{d}_{n+1}V @VV\bar{d}_{n+1}V @VV\bar{d}_{n+1}V @. \\
 0 @>>> P_n @>>> F_n @>>> Y_n  @>>> 0 \\
 @. @VV\pi_nV @VV\pi_nV @VV\pi_nV @. \\
  0 @>>> \frac{P_n}{B_nP} @>>> \frac{F_n}{B_nF} @>>> \frac{Y_n}{B_nY}  @>>> 0 \\
  @. @VVV @VVV @VVV @. \\
  @. 0 @. 0 @. 0 \\
\end{CD}$$
Hence applying $A \tensor_R -$ to this diagram, the rows and columns all remain exact. Moreover, in light of part~(3) of Lemma~\ref{lemma-A-acyclic}, the middle column enjoys that  $0 \xrightarrow{} A \otimes_R \frac{F_{n+1}}{B_{n+1}F} \xrightarrow{1_A \otimes \bar{d}_{n+1}} A \otimes_R F_n$ is a monomorphism.
We can now argue using the snake lemma that $0 \xrightarrow{} A \otimes_R \frac{P_{n+1}}{B_{n+1}P} \xrightarrow{1_A \otimes \bar{d}_{n+1}} A \otimes_R P_n$ and $0 \xrightarrow{} A \otimes_R \frac{Y_{n+1}}{B_{n+1}Y} \xrightarrow{1_A \otimes \bar{d}_{n+1}} A \otimes_R Y_n$ are also monomorphisms. By again applying part~(3) of Lemma~\ref{lemma-props of pure subcomplexes} this means $P$ and $Y$ are in ${}_A\tilclass{F}$, proving ${}_A\tilclass{F}$ is closed under pure subcomplexes and pure quotients.

It now follows from Proposition~\ref{prop-transfinite} that there exists a regular cardinal $\kappa$ such that every complex $F \in {}_A\tilclass{F}$ is a transfinite extension of the class of all complexes in ${}_A\tilclass{F}$ with cardinality bounded by $\kappa$. So we can let $\class{S}$ be a set of representatives for the isomorphism classes of all complexes in ${}_A\tilclass{F}$ with cardinality bounded by $\kappa$. Then $\class{S}$ cogenerates a complete cotorsion pair $(\leftperp{(\rightperp{\class{S}})}, \rightperp{\class{S}})$, and $\leftperp{(\rightperp{\class{S}})}$ consists precisely of direct summands of transfinite extensions of complexes in $\class{S}\cup\{D^n(R)\}$.
But clearly $\class{S}\cup\{D^n(R)\} \subseteq {}_A\tilclass{F}$. Also, ${}_A\tilclass{F}$ is clearly closed under direct summands (which in fact are pure), extensions, and direct limits, and hence all transfinite extensions.
We conclude $\leftperp{(\rightperp{\class{S}})} = {}_A\tilclass{F}$, proving (${}_A\tilclass{F}, \rightperp{{}_A\tilclass{F}})$ is a cotorsion pair cogenerated by a set.

It is now a standard result that the cotorsion pair is complete and that every complex $X$ has a ${}_A\tilclass{F}$-cover, since ${}_A\tilclass{F}$ is closed under direct limits. See~\cite[Corollary~5.2.7]{enochs-jenda-book} for an argument for module categories that carries over to Grothendieck categories.
\end{proof}

\subsection{AC-acyclic and firmly acyclic}
Based on the duality between absolutely clean and level modules, the following new type of acyclicity for complexes of projectives became interesting in~\cite{bravo-gillespie-hovey}. For our purposes we are now extending the definitions to complexes of flats.

\begin{definition}\label{def-acyclicity}
Let $X$ be a chain complex of flat $R$-modules, including perhaps  a complex of projectives.
\begin{enumerate}
\item We say that $X$ is \textbf{AC-acyclic} if $A \tensor_R X$ is exact for every absolutely clean right $R$-module $A$.  If $X$ is itself exact we say $X$  is \textbf{exact AC-acyclic}.
\item We say that $X$ is  \textbf{firmly acyclic} if $\Hom_R(X,F)$ is exact for all level left $R$-modules $F$. If $X$ is itself exact we say $X$ is \textbf{exact firmly acyclic}.
\end{enumerate}
\end{definition}

A main result of~\cite{bravo-gillespie-hovey}, extending a result of Murfet and Salarian from~\cite{murfet-salarian}, is the following.

\begin{theorem}[\cite{bravo-gillespie-hovey}, Theorem~6.6]\label{them-AC-acyclic-firmly-acyclic}
Let $R$ be any ring and $P$ a complex of projective modules. Then $P$ is AC-acyclic if and only if it is firmly acyclic. If all level modules have finite projective dimension, these conditions are equivalent to $\Hom_R(P,Q)$ being exact for all projective modules $Q$.
\end{theorem}

\begin{remark}\label{remark-von Neumann}
 It is important to note that $X$ must be taken to be a complex of projectives in the above theorem. Indeed
there are von Neumann regular rings (every module is flat) that are not semisimple (every module is projective). Such a ring is coherent and there must be short exact sequences that are not split, though all short exact sequences are pure. Any such sequence must be exact AC-acyclic when viewed as a chain complex. However, it can't be exact firmly acyclic because then it would have to split.
\end{remark}

The main motivation for considering $A$-acyclic complexes is explained in the following setup.

\begin{setup}\label{setup-choice of A} We know from~\cite[Proposition~2.5~(v)]{bravo-gillespie-hovey} that there is a set $\class{S}$ of modules for which every absolutely clean (right) $R$-module is a transfinite extension of modules in $\class{S}$. Fix such a set $\class{S}$ and let $A$ to be the direct sum of all modules in $\class{S}$. Then a complex $X$ of flat modules is $A$-acyclic precisely when it is AC-acyclic in the sense of Definition~\ref{def-acyclicity}. If we wish, we may also ``throw in'' the (right) module $R$, making it too a summand of $A$. In this case $X$ is $A$-acyclic precisely when it is an exact AC-acyclic complex.
\end{setup}

Taking $A$ to be as in Setup~\ref{setup-choice of A}, we get the following from Theorem~\ref{them-deconstructibility of tensor acyclic complexes}.

\begin{corollary}\label{cor-AC acyclic}
Let ${}_A\tilclass{F}$ denote the class of all AC-acyclic complexes of flat modules in $\ch$. Then $({}_A\tilclass{F}, \rightperp{{}_A\tilclass{F}})$ is a complete hereditary cotorsion pair cogenerated by a set. Moreover, every chain complex has a surjective ${}_A\tilclass{F}$-cover. The same results hold if we replace ${}_A\tilclass{F}$ with the class of all exact AC-acyclic complexes of flat modules.
\end{corollary}

\section{Flat model structures on complexes of modules}\label{sec-AC-acyclic flat model on ch}

Let $R$ be a ring and $A$ a given $R$-module. In~\cite[Theorem~4.1]{bravo-gillespie-hovey} we see a general theorem that puts an injective model structure, the \emph{$A$-acyclic injective model structure}, on $\ch$. Its main application was to put a model structure on $\ch$ whose fibrant objects are all the exact AC-acyclic complexes of injectives; that is, exact complexes $I$ of injectives such that $\Hom_R(A,I)$ remains exact for all absolutely clean (left) $R$-modules $A$. Denoting its homotopy category by $\cat{S}_{inj}$, we think of this category as the injective stable category of $R$. On the other hand, we then see in~\cite[Theorem~6.1]{bravo-gillespie-hovey} a general theorem that puts a projective model structure, the \emph{$A$-acyclic projective model structure}, on $\ch$. Its main application was to put a model structure on $\ch$ whose cofibrant objects are all the exact firmly acyclic complexes of projectives; that is, exact complexes $P$ of projectives such that $\Hom_R(P,L)$ remains exact for all level (left) $R$-modules $L$. Denoting its homotopy category by $\cat{S}_{prj}$, we think of this category as the projective stable category of $R$. In general, $\cat{S}_{inj}$ is not equivalent to $\cat{S}_{prj}$, but they agree when $R$ is Gorenstein and recover the usual stable module category in this case.

It is natural to attempt to make similar constructions using complexes of flat modules, and that is the goal of this section. As it turns out, the flat model structures of this section coincide with the projective model structures of~\cite[Section~6]{bravo-gillespie-hovey}, in the sense that they have the same homotopy categories.
This means that we can ``mock'' the projective model structures for (quasi-coherent) sheaf categories. In particular, $\cat{S}_{prj}$ ought to extend to sheaf categories, and this will be the subject of Section~\ref{sec-sheaves} and~\ref{sec-AC-acyclic flat model on sheaf}.

For comparison purposes we now restate the relevant result from~\cite{bravo-gillespie-hovey}.

\begin{theorem}[\cite{bravo-gillespie-hovey}, Theorem~6.1]\label{thm-how to create projective on chain}
Let $R$ be any ring and $A$ a given right $R$-module.  Let
${}_A\tilclass{P}$ be the class of all $A$-acyclic complexes of projectives; that
is, chain complexes $P$ that are degreewise projective and such that
$A\otimes_{R} P$ is exact.  Then there is a cofibrantly generated
abelian model structure on $\ch$ determined by the projective cotorsion pair
$({}_A\tilclass{P}, \class{W})$ where $\class{W} = \rightperp{({}_A\tilclass{P})}$.
Furthermore, $\class{W}$ contains all
contractible complexes.  We call this model structure the
\textbf{$A$-acyclic projective model structure}.  Its homotopy category
is equivalent to $K_{\text{A-ac}}(Proj)$, the chain homotopy category of all $A$-acyclic complexes
of projectives.
\end{theorem}

The next theorem says that the homotopy category $K_{\text{A-ac}}(Proj)$ can be recovered from a model structure based on flat modules.  First we recall some notation from~\cite{gillespie}. Given a cotorsion pair $(\class{F},\class{C})$ in $R$-Mod, it lifts to two cotorsion pairs $(\tilclass{F},\dgclass{C})$ and $(\dgclass{F},\tilclass{C})$ in $\ch$. Here $\tilclass{F}$ (resp. $\tilclass{C}$) consists of all exact complexes $X$ with cycles $Z_nX \in \class{F}$ (resp.$Z_nX \in \class{C}$). We also let $\class{E}$ denote the class of all exact complexes. When $(\class{F},\class{C})$ is Enochs' flat cotorsion pair, then it was shown in~\cite{gillespie} that we have an hereditary abelian model structure $(\dgclass{F}, \class{E}, \dgclass{C})$ on $\ch$ whose corresponding complete cotorsion pairs are $(\tilclass{F},\dgclass{C})$ and $(\dgclass{F},\tilclass{C})$. We call this the \emph{flat model structure} for the derived category $\class{D}(R)$. The complexes in $\tilclass{F}$ are the flat objects in $\ch$; they are direct limits of projective complexes.

\begin{theorem}\label{thm-how to create flat on chain}
Let $R$ be any ring and $A$ a given right $R$-module.
Let ${}_A\tilclass{F}$ be the class of all $A$-acyclic complexes of flat modules; that
is, chain complexes $F$ that are degreewise flat and such that
$A\otimes_{R} F$ is exact. Then there is a cofibrantly generated abelian model structure, the \textbf{A-acyclic flat model structure},  $({}_A\tilclass{F}, \class{W}, \dgclass{C})$ on $\ch$. Moreover, the thick class $\class{W}$ of trivial objects is exactly the same $\class{W}$ as in Theorem~\ref{thm-how to create projective on chain}.
Thus its homotopy category is also equivalent to $K_{\text{A-ac}}(Proj)$, the chain homotopy category of all $A$-acyclic complexes
of projectives.
\end{theorem}

\begin{proof}
We already have that $(\tilclass{F}, \dgclass{C})$ is a complete hereditary cotorsion pair, cogenerated by a set, and the same is true for $({}_A\tilclass{F}, \rightperp{{}_A\tilclass{F}})$ by Theorem~\ref{them-deconstructibility of tensor acyclic complexes}. We have $\tilclass{F} \subseteq {}_A\tilclass{F}$, and so by~\cite{gillespie-hovey triples} we immediately get an hereditary abelian model structure $({}_A\tilclass{F}, \class{W}, \dgclass{C})$ once we show that ${}_A\tilclass{F} \cap \rightperp{{}_A\tilclass{F}} =  \tilclass{F} \cap \dgclass{C}$. But by~\cite[Proposition~3.7]{gillespie-degreewise-model-strucs} we get that $\rightperp{{}_A\tilclass{F}}$ (resp. $\dgclass{C}$) is precisely the class of all complexes of cotorsion modules $C$ such that any map $F \xrightarrow{} C$ with $F \in {}_A\tilclass{F}$ (resp. $F \in \tilclass{F}$) is null homotopic.  We conclude that ${}_A\tilclass{F} \cap \rightperp{{}_A\tilclass{F}}$ (resp. $\tilclass{F} \cap \dgclass{C}$) coincides with the class of all contractible complexes of flat cotorsion modules because any $X \in {}_A\tilclass{F} \cap \rightperp{{}_A\tilclass{F}}$ (resp. $X \in \tilclass{F} \cap \dgclass{C}$) must be a complex of flat cotorsion modules with $1_X \sim 0$.

Now it is a standard fact from~\cite{hovey} that the model structure is cofibrantly generated since the two associated cotorsion pairs are each cogenerated by a set. So it is left to prove that the thick class $\class{W}$ of trivial objects is exactly the same $\class{W}$ from Theorem~\ref{thm-how to create projective on chain}.
This will follow immediately from~\cite[Proposition~3.2]{gillespie-recollement} and~\cite[Lemma~2.3(1)]{gillespie-injective models} once we show ${}_A\tilclass{F} \cap \rightperp{({}_A\tilclass{P})} = \tilclass{F}$.  But ${}_A\tilclass{F} \cap \rightperp{({}_A\tilclass{P})} \supseteq \tilclass{F}$ is clear from Neeman's result, Lemma~\ref{lemma-Neeman} below. Using that same lemma, we now show ${}_A\tilclass{F} \cap \rightperp{({}_A\tilclass{P})} \subseteq \tilclass{F}$. So let $F \in {}_A\tilclass{F} \cap \rightperp{({}_A\tilclass{P})}$. From~\cite{bravo-gillespie-hovey} we have a complete cotorsion pair $(\dwclass{P}, \rightperp{(\dwclass{P})})$, where $\dwclass{P}$ is the class of all complexes of projectives. So we may write a short exact sequence $$0 \xrightarrow{} F  \xrightarrow{} W  \xrightarrow{} P  \xrightarrow{} 0 $$ with $W \in  \rightperp{(\dwclass{P})}$ and $P \in \dwclass{P}$. But then using Lemma~\ref{lemma-Neeman}, one easily argues that $W \in \tilclass{F}$. Since the short exact sequence is degreewise split we obtain a short exact sequence   $$0 \xrightarrow{} A \tensor_R F  \xrightarrow{} A \tensor_R W  \xrightarrow{} A \tensor_R P  \xrightarrow{} 0. $$  We have that $A \tensor_R F$  and $A \tensor_R W$ are each exact and so it follows  that $A \tensor_R P$ is also exact; that is, $P \in {}_A\tilclass{P}$. Therefore the original short exact sequence must split. This forces $F \in \tilclass{F}$ and completes the proof that ${}_A\tilclass{F} \cap \rightperp{({}_A\tilclass{P})} = \tilclass{F}$.
\end{proof}

The above proof heavily relies on the following nontrivial result of Neeman.

\begin{lemma}[Neeman~\cite{neeman-flat}]\label{lemma-Neeman}
Let $R$ be any ring, $\dwclass{P}$ (resp. $\dwclass{F}$) denote the class of all complexes of projectives (resp. flats), and $\tilclass{F}$ denote the class of all categorically flat complexes; that is, exact complexes with all flat cycles. Then $\dwclass{F} \cap \rightperp{(\dwclass{P})} = \tilclass{F}$.
\end{lemma}

\subsection{The exact AC-acyclic flat model structure}
We now consider again the special case of taking $A$ to be the module from Setup~\ref{setup-choice of A}. Then a complex of flats is $A$-acyclic precisely when it is (exact) AC-acyclic in the sense of Definition~\ref{def-acyclicity}. Moreover,  by Theorem~\ref{them-AC-acyclic-firmly-acyclic}, for complexes of projectives this is equivalent to it being (exact) firmly acyclic. Note that if $R$ is a ring for which all level modules have finite projective dimension, an exact complex of projectives satisfies these properties if and only if it is \textbf{totally acyclic} in the usual sense that it remains exact after applying $\Hom_R(-,Q)$ for any projective module $Q$.
So taking $A$ as in Setup~\ref{setup-choice of A}, the following result was obtained in~\cite{bravo-gillespie-hovey}. Again, we repeat the statement here for the convenience in comparing it to Corollary~\ref{cor-exact-AC-flat-model}.

\begin{corollary}[\cite{bravo-gillespie-hovey}, Theorem~6.7]\label{cor-firmly-acyclic}
Let $R$ be any ring and ${}_A\tilclass{P}$ be the class of all firmly acyclic complexes of projectives, or equivalently, all AC-acyclic complexes of projectives. (See Definition~\ref{def-acyclicity} and Theorem~\ref{them-AC-acyclic-firmly-acyclic}.)
Then there is a cofibrantly generated
abelian model structure on $\ch$ determined by the projective cotorsion pair
$({}_A\tilclass{P}, \class{W})$ where $\class{W} = \rightperp{({}_A\tilclass{P})}$.
The same statements hold by letting ${}_A\tilclass{P}$ be the class of all exact firmly acyclic complexes of projectives, or equivalently, all exact AC-acyclic complexes of projectives.
We call these model structures the
\textbf{(exact) firmly acyclic model structure}, or, alternatively, the
\textbf{(exact) AC-acyclic projective model structure}.  Their homotopy categories
are equivalent to the chain homotopy category of all (exact) firmly acyclic complexes of projectives.
\end{corollary}

\begin{definition}\label{def-firm-hocat}
Let $K(R)$ denote the usual chain homotopy category of complexes. Let $K(R)/\class{W}$ be the Verdier quotient by the thick subcategory $\class{W} = \rightperp{({}_A\tilclass{P})}$ of all trivial objects in the exact firmly acyclic model structure. Then Corollary~\ref{cor-firmly-acyclic} tells us $K(R)/\class{W} \cong K_{fir}(Proj)$, where $K_{fir}(Proj)$ denotes the chain homotopy category of all exact firmly acyclic complexes of projectives. If all level $R$-modules have finite projective dimension this coincides with $K_{tac}(Proj)$, the chain homotopy category of all totally acyclic complexes of projectives.
\end{definition}

Keeping the same choice of module $A$, but instead applying Theorem~\ref{thm-how to create flat on chain}, we get the following corollary. We think of it as saying that we can ``mock''  $K_{fir}(Proj)$, and hence in some cases $K_{tac}(Proj)$, using just the flat modules.

\begin{corollary}\label{cor-exact-AC-flat-model}
Let $R$ be any ring and ${}_A\tilclass{F}$ be the class of all AC-acyclic complexes of flat modules. Then there is a cofibrantly generated abelian model structure $({}_A\tilclass{F}, \class{W}, \dgclass{C})$  on $\ch$. Moreover, the thick class $\class{W}$ of trivial objects is exactly the same $\class{W}$ as in the firmly acyclic model structure of Corollary~\ref{cor-firmly-acyclic}. The same statements hold if we instead let ${}_A\tilclass{F}$ denote the class of all exact AC-acyclic complexes of flat modules. We call these model structures the \textbf{(exact) AC-acyclic flat model structure}.
\end{corollary}

\begin{corollary}\label{cor-flat-model}
Let $R$ be any ring. Then we have triangle equivalences:
$$K_{fir}(Proj) \,\cong \, K(R)/\class{W} \, \cong \, ({}_A\tilclass{F} \cap \dgclass{C})/\sim$$
where ${}_A\tilclass{F}$ are the exact AC-acyclic complexes of flat modules, $\dgclass{C} = \rightperp{\tilclass{F}}$ are the cotorsion complexes, and $\sim$ is the usual chain homotopy relation.
When all level $R$-modules have finite projective dimension these coincide with $K_{tac}(Proj)$.
\end{corollary}

\section{The coherent case}\label{sec-coherent case}

Now let $R$ be a coherent ring. In this section we show that there are four different abelian model structures, all Quillen equivalent, serving as models for the \emph{projective stable module category of $R$}. For easy reference throughout this section, the four model structures are highlighted below. The main result is Theorem~\ref{them-diagram}.

The \emph{\textbf{exact firmly acyclic model structure}} on $\ch$ from~\cite[Theorem~6.7]{bravo-gillespie-hovey}, as recalled in Corollary~\ref{cor-firmly-acyclic}. It is represented by the projective cotorsion pair $({}_A\tilclass{P}, \class{W})$ where  ${}_A\tilclass{P}$ is the class of all exact firmly acyclic complexes of projectives. Since we assume $R$ is coherent these are exact complexes $P$ of projectives such that $\Hom_R(P,F)$ remains exact for all flat modules $F$. By~\cite[Theorem~6.6]{bravo-gillespie-hovey}, this is equivalent to $A \tensor_R P$ remaining exact for all absolutely pure (right) $R$-modules $A$. In this section we will let $\ch_{proj}$ denote this model structure.

The \emph{\textbf{exact AC-acyclic flat model structure}} on $\ch$, from Corollary~\ref{cor-exact-AC-flat-model}, represented by the Hovey triple $({}_A\tilclass{F}, \class{W}, \dgclass{C})$.  Since $R$ is coherent ${}_A\tilclass{F}$ consists of exact complexes $F$ of flat modules such that $A \tensor_R F$ remains exact for all absolutely pure (right) $R$-modules $A$. In this section we will let $\ch_{flat}$ denote this model structure.

The \emph{\textbf{Gorenstein AC-projective model structure}} on $R$-Mod, from~\cite{bravo-gillespie-hovey}, represented by the projective cotorsion pair $(\class{GP},\rightperp{\class{GP}})$. Here $\class{GP}$ is the class of all Gorenstein AC-projective modules. But since we are assuming $R$ is a coherent ring, these are exactly the Ding projective modules of~\cite{gillespie-Ding-Chen rings}. By definition, they are modules $M$ which appear as $M = Z_0P$, where $P$ is an exact complex of projectives such that $\Hom_R(P,F)$ remains exact for all flat modules $F$. In other words, $M$ is a cycle module of an exact firmly acyclic complex of projectives. In this section we will let  $R\textnormal{-Mod}_{proj}$ denote this model structure.

From~\cite{gillespie-flat stable}, we have the \emph{\textbf{Gorenstein flat model structure}} on $R$-Mod. This is an hereditary abelian model structure represented by a Hovey triple $(\class{GF},\class{W},\class{C})$ where $\class{GF}$ is the class of Gorenstein flat modules and $\class{C} = \rightperp{\class{F}}$ are the cotorsion modules.  Here $\class{F}$ is the class of flat modules and we also let $\class{GC} =  \rightperp{\class{GF}}$ denote the class of Gorenstein cotorsion modules. $\class{W}$ is the smallest thick class of modules containing $\class{F}$ and $\class{GC}$ and it satisfies $\class{GF} \cap \class{W} = \class{F}$ and $\class{W} \cap \class{C} = \class{GC}$.  In this section we will let  $R\textnormal{-Mod}_{flat}$ denote this model structure.

\begin{theorem}\label{them-diagram}
Let $R$ be a coherent ring and let $F$ be the functor $F : \ch \xrightarrow{} R\textnormal{-Mod}$ given by $X \mapsto X_0/B_0X$. Then all four functors in the commutative diagram below are Quillen equivalences:
$$
\begin{CD}
\ch_{proj} @>F>> R\textnormal{-Mod}_{proj} \\
@V1_{\ch}VV @VV 1_{R\textnormal{-Mod}}V \\
\ch_{flat} @>F>> R\textnormal{-Mod}_{flat}
\end{CD}$$
In particular, the homotopy categories of all four model structures are equivalent.
\end{theorem}

We start with the following result, whose proof is similar to that of Theorem~\ref{thm-how to create flat on chain}, relying on Neeman's Lemma~\ref{lemma-Neeman}. But it also relies on Lemma~\ref{lemma-AP-acyclic} which appears afterwards.

\begin{proposition}\label{prop-trivial-objects}
Let $R$ be a coherent ring. Then the trivial objects $\rightperp{\class{GP}}$ of the Gorenstein AC-projective model structure $(\class{GP},\rightperp{\class{GP}})$ (of~\cite{bravo-gillespie-hovey}) coincide with the trivial objects $\class{W}$ in the Gorenstein flat model structure $(\class{GF},\class{W},\class{C})$ (of~\cite{gillespie-flat stable}).
\end{proposition}

\begin{proof}
As already pointed out, since $R$ is coherent, a module $M$ is Gorenstein AC-projective (Ding projective) if and only if $M = Z_0P$, where $P$ is an exact complex of projectives such that $A \tensor_R P$ is exact for all absolutely pure (right) $R$-modules $A$. If follows from Lemma~\ref{lemma-AP-acyclic} that $M$ is Gorenstein flat. So $\class{GP} \subseteq \class{GF}$. Hence $\class{GC} \subseteq \rightperp{\class{GP}}$, where $\class{GC} = \rightperp{\class{GF}}$ are the Gorenstein cotorsion modules. Also, $\rightperp{\class{GP}}$ is already known to be thick from~\cite{bravo-gillespie-hovey}. With these observations, the claim will follow immediately from~\cite[Proposition~3.2]{gillespie-recollement} combined with~\cite[Lemma~2.3(1)]{gillespie-injective models} once we show $\class{GF} \cap \rightperp{\class{GP}} = \class{F}$, where $\class{F}$ is the class of all flat modules.

It is easy to see that $\class{GF} \cap \rightperp{\class{GP}} \supseteq \class{F}$, so we focus on
$\class{GF} \cap \rightperp{\class{GP}} \subseteq \class{F}$. Let $M \in \class{GF} \cap \rightperp{\class{GP}}$, and write it as $M = Z_0F$ where $F$ is a complete flat resolution of $M$. That is, $I \tensor_R F$ remains exact for all injectives $I$. But again from Lemma~\ref{lemma-AP-acyclic} it is true that $A \tensor_R F$ remains exact for all absolutely pure modules $A$. From~\cite{bravo-gillespie-hovey} we have a complete cotorsion pair $(\dwclass{P}, \rightperp{(\dwclass{P})})$, where $\dwclass{P}$ is the class of all complexes of projectives. So we may write a short exact sequence $$0 \xrightarrow{} F  \xrightarrow{} W  \xrightarrow{} P  \xrightarrow{} 0 $$ with $W \in  \rightperp{(\dwclass{P})}$ and $P \in \dwclass{P}$. But then using Lemma~\ref{lemma-Neeman}, one easily argues that $W \in \tilclass{F}$, the class of all exact complexes with all cycle modules flat. Since $F$ and $W$ are each exact, we see that $P$ is exact too. Moreover,  the short exact sequence is split in each degree and so for any absolutely pure $A$ we have a short exact sequence   $$0 \xrightarrow{} A \tensor_R F  \xrightarrow{} A \tensor_R W  \xrightarrow{} A \tensor_R P  \xrightarrow{} 0. $$ It is now clear that $A \tensor_R P$ is also exact, equivalently $\Hom_R(P,F)$ is exact for all flats, and so $Z_0P$ is Ding projective. Note that since each complex is exact we have a short exact sequence $0 \xrightarrow{} Z_0F  \xrightarrow{} Z_0W  \xrightarrow{} Z_0P  \xrightarrow{} 0$. By the hypothesis, $Z_0F \in \rightperp{\class{GP}}$, and so we conclude  that this sequence splits. Since $Z_0W$ is flat, so is the direct summand $Z_0F$, proving $\class{GF} \cap \rightperp{\class{GP}} \subseteq \class{F}$.
\end{proof}

The following lemma is crucial to the above proof but also for what is still to come. It goes back to~\cite[Lemma~2.8]{ding and mao 08}. However we provide a new proof which is quick and easy.

\begin{lemma}\label{lemma-AP-acyclic}
Let $R$ be any ring and $X$ a (not necessarily exact) complex of flat modules. Then $A \otimes_R X$ is exact for all absolutely pure (right) $R$-modules $A$ if and only if $I \otimes_R X$ is exact for all injective (right) $R$-modules $I$.
\end{lemma}

\begin{proof}
Note that the ``only if'' part is trivial. Conversely, suppose that $$X = \cdots \rightarrow X_1 \rightarrow X_0 \rightarrow X^0 \rightarrow X^1 \rightarrow \cdots$$ is a complex of flat modules which remains exact after applying $I \otimes_R -$ for any injective $I$.  We must show that this complex does in fact remain exact after applying $A \otimes_R -$ for any absolutely pure $A$. So let such an $A$ be given and note that by Lemma~\ref{lemma-A-acyclic} it is equivalent to show that, for each $n$, the map below is a monomorphism:
$$A \otimes_R \frac{X_{n+1}}{B_{n+1}X} \xrightarrow{} A \otimes_R X_n$$
Now let $A \hookrightarrow I$ be an embedding into an injective module $I$ and note that this must be a pure monomorphism since $A$ is a absolutely pure. For each $n$ we have the commutative diagram:
$$\begin{CD}
A \otimes_R \frac{X_{n+1}}{B_{n+1}X}  @>>> I \otimes_R \frac{X_{n+1}}{B_{n+1}X} \\
@VVV @VVV       \\
A \otimes_R X_n @>>> I \otimes_R X_n \\
\end{CD}$$
The two horizontal arrows are monomorphisms since $A \hookrightarrow I$ is pure. The right vertical arrow is also a monomorphism since $I \otimes_R X$ is exact.  It follows that the left vertical arrow must also be a monomorphism. This completes the proof. (The last statement of the lemma holds since for coherent rings the absolutely clean modules coincide with the absolutely pure modules.)
\end{proof}

It now follows from the following simple lemma that the identity functors making up the vertical arrows of the diagram in Theorem~\ref{them-diagram} are Quillen equivalences.

\begin{lemma}\label{lemma-identity}
Let $\class{A}$ be an abelian category with two model structures $\class{M} = (\class{Q}, \class{W}, \class{R})$ and $\class{M}' = (\class{Q}', \class{W}, \class{R}')$, each with the same class  $\class{W}$ of trivial objects. If $\class{Q} \subseteq \class{Q}'$, then the identity functor $1_{\class{A}}$ is a (left) Quillen equivalence from $\class{M}$ to $\class{M}'$. Equivalently, if $\class{R}' \subseteq \class{R}$,  then the identity functor $1_{\class{A}}$ is a right Quillen equivalence from $\class{M}'$ to $\class{M}$.
\end{lemma}

\begin{proof}
 Note that $1_{\class{A}}$ is certainly left adjoint to itself. It is easy to see that $\class{Q} \subseteq \class{Q}'$ if and only if $\class{R}' \subseteq \class{R}$ since $\class{W}$ is the same in both model structures. So the two statements are equivalent. Now if $\class{Q} \subseteq \class{Q}'$, then $1_{\class{A}}$ from $\class{M}$ to $\class{M}'$ preserves cofibrant and trivially cofibrant objects. Hence it preserves all cofibrations and trivial cofibrations, making it a left Quillen functor.  Now in any abelian model structure, a map $f$ is a weak equivalence if and only if it factors as a monomorphism with trivial cokernel followed by an epimorphism with trivial kernel.~\cite[Lemma~5.8]{hovey}. It now follows from the very definition, see~\cite[Definition~1.3.12]{hovey-model-categories}, that $1_{\class{A}}$ is a Quillen equivalence.
\end{proof}

We now turn to the horizontal functor $F$ in Theorem~\ref{them-diagram}. It was already shown in~\cite[Theorem~8.8]{bravo-gillespie-hovey} that  $F$ in the top of the diagram is a Quillen equivalence. We now show that it extends to a Quillen equivalence between the flat model structures in the bottom of the diagram.

\begin{proposition}\label{prop-equiv}
Let $R$ be a coherent ring. The functor $F : \ch \xrightarrow{} R\text{-Mod}$ defined by $F(X) = X_0/B_0X$ is a Quillen equivalence from the exact AC-acyclic flat model structure (of Corollary~\ref{cor-exact-AC-flat-model}) to the Gorenstein flat model structure (of~\cite{gillespie-flat stable}).
\end{proposition}

\begin{proof}
Note that for any ring $R$, the functor $F$ is a left adjoint. Its right adjoint is the functor $S^0$ which turns a module into a complex concentrated in degree zero.
Recall that a left Quillen functor is a left adjoint that preserves cofibrations and trivial cofibrations. To see that $F$ is left Quillen, let $X \xrightarrow{f} Y$ be a cofibration in the exact AC-acyclic flat model structure. Then by definition we have a short exact sequence $0 \xrightarrow{} X  \xrightarrow{f} Y \xrightarrow{} C \xrightarrow{} 0$ with $C$ an exact AC-acyclic complex of flat modules. The functor $F$ is only left exact in general, but since $C$ is exact we do get a short exact sequence $0 \xrightarrow{} X_0/B_0X  \xrightarrow{F(f)} Y_0/B_0Y \xrightarrow{} C_0/B_0C \xrightarrow{} 0$. We have $C_0/B_0C \cong Z_{-1}C$ which is Gorenstein flat by Lemma~\ref{lemma-AP-acyclic}. Since $F(f)$ is a monomorphism with Gorenstein flat cokernel, $F$ preserves cofibrations. The same argument shows that $F$ preserves trivial cofibrations; we just replace $C$ with a categorically flat complex, which is one with flat cycles.

To show that $F$ is a Quillen equivalence we use~\cite[Corollary~1.3.16(b)]{hovey-model-categories}. In our case it means we must prove the following: (i)  If $X \xrightarrow{f} Y$ is a chain map between two exact AC-acyclic complexes of flats for which $X_0/B_0X \xrightarrow{F(f)} Y_0/B_0Y$ is a weak equivalence, then $f$ itself must be a weak equivalence. (ii) For all cotorsion modules $C$, the map $FQS^0(C) \xrightarrow{} C$, where $Q$ is cofibrant replacement, is a weak equivalence.

To prove (i), we use the factorization axiom to write $f = pi$ where $X \xrightarrow{i} Z$ is a trivial cofibration and $Z \xrightarrow{p} Y$ is a fibration. We note that we have short exact sequences $0 \xrightarrow{} X  \xrightarrow{i} Z \xrightarrow{} C \xrightarrow{} 0$ and  $0 \xrightarrow{} K  \xrightarrow{} Z \xrightarrow{p} Y \xrightarrow{} 0$; since $C$ is categorically flat it is exact AC-acyclic and hence $Z$ and therefore $K$ must be too. Applying $F$ to this factorization gives us short exact sequences
$0 \xrightarrow{} X_0/B_0X  \xrightarrow{F(i)} Z_0/B_0Z \xrightarrow{} C_0/B_0C \xrightarrow{} 0$ and  $0 \xrightarrow{} K_0/B_0K  \xrightarrow{} Z_0/B_0Z \xrightarrow{F(p)} Y_0/B_0Y \xrightarrow{} 0$ and a factorization $F(f) = F(p)F(i)$. As already shown above, $F(i)$ is a trivial cofibration. Also $F(f)$ is a weak equivalence by hypothesis. So by the two out of three axiom, $F(p)$ must also be a weak equivalence. Being a surjection, it means that $\ker{(F(p))} = K_0/B_0K$ must be trivial. But since $K$ is exact AC-acyclic, $\ker{(F(p))}$ must also be Gorenstein flat by Lemma~\ref{lemma-AP-acyclic}.  This means $\ker{(F(p))}$ is trivially cofibrant; that is, a flat module. But the class of trivial modules is thick and contains the flat modules, and hence each $Z_nK$  must be trivially cofibrant (flat). This proves that $\ker{p}$ is a flat complex, and so $p$ is a trivial fibration.

To prove (ii), let $C$ be a cotorsion module. To get a cofibrant replacement of $S^0C$ in the exact AC-acyclic flat model structure,  we use enough projectives of the cotorsion pair
$({}_A\tilclass{F}, \rightperp{{}_A\tilclass{F}})$. This gives us a short exact sequence
$0 \xrightarrow{} Y \xrightarrow{} F \xrightarrow{} S^0C \xrightarrow{} 0$ with $F$ an exact AC-acyclic complex of flats and $Y \in \rightperp{{}_A\tilclass{F}}$. Then by~\cite[Lemma~8.1]{bravo-gillespie-hovey} we get that $Y_0/B_0Y \in \rightperp{\class{GP}}$, where $\class{GP}$ are the Gorenstein AC-projective modules. Applying $F$ to the short exact sequence gives us, using the snake lemma, another short exact sequence $0 \xrightarrow{} Y_0/B_0Y \xrightarrow{} F_0/B_0F \xrightarrow{} C \xrightarrow{} 0$. The problem is to show that $F_0/B_0F \xrightarrow{} C$ is a weak equivalence. Being an epimorphism, it is enough to show its kernel is trivial in the Gorenstein flat model structure. But this is indeed the case, as we show directly in Proposition~\ref{prop-trivial-objects} that $\rightperp{\class{GP}}$ is the class of trivial objects in this model structure.
\end{proof}

Putting this all together we now have proved Theorem~\ref{them-diagram}:

\begin{proof}[Proof of Theorem~\ref{them-diagram}]
The functor $F$ in the top of the diagram is a Quillen equivalence from the \emph{exact firmly acyclic model structure} to the \emph{Gorenstein AC-projective model structure} by~\cite[Theorem~8.8]{bravo-gillespie-hovey}. By Proposition~\ref{prop-equiv}, $F$ is also a Quillen equivalence from the \emph{exact AC-acyclic flat model structure} to the \emph{Gorenstein flat model structure} of~\cite{gillespie-flat stable}. The left vertical identity functor is a Quillen equivalence by Lemma~\ref{lemma-identity} and Corollary~\ref{cor-exact-AC-flat-model}. Similarly, the right vertical identity functor is a Quillen equivalence by Lemma~\ref{lemma-identity} and Proposition~\ref{prop-trivial-objects}.
\end{proof}

\section{Acyclicity for complexes of flat quasi-coherent sheaves}\label{sec-sheaves}

We now let $X$ be a scheme with structure sheaf $\mathscr{R}$. A ``sheaf'' will always mean a quasi-coherent sheaf. That is, all sheaves will be assumed to be quasi-coherent without further mention, and $\Qco(X)$ will denote the category of all such sheaves on $X$. We let Ch$(\Qco(X))$ denote the associated category of (unbounded) chain complexes. It is known that $\Qco(X)$ and Ch$(\Qco(X))$ are both Grothendieck categories; for example, see~\cite{ee-quasi-coherent}. As such they are also known to be locally $\lambda$-presentable, in the sense of~\cite{adamek-rosicky}, for some regular cardinal $\lambda$.

\subsection{Preliminary results on complexes of sheaves} The following definition of purity in Ch$(\Qco(X))$ will prove useful for our purposes. 

\begin{definition}\label{def-pur}
Let $\mathscr{F}_{\bullet}$ be a chain complex of sheaves. A subcomplex $\mathscr{P}_{\bullet}\subseteq \mathscr{F}_{\bullet}$ will be called a \textbf{pure subcomplex} if for every open affine $U\subseteq X$, the subcomplex $\mathscr{P}_{\bullet}(U)\subseteq \mathscr{F}_{\bullet}(U)$ of $\mathscr{R}(U)$-modules is a pure subcomplex in the sense of Definition~\ref{def-purity}.
We will also say that $\mathscr{F}_{\bullet}/\mathscr{P}_{\bullet}$ is a \textbf{pure quotient} and refer to
a short exact sequence $0\to \mathscr{P}_{\bullet}\to \mathscr{F}_{\bullet}\to \mathscr{F}_{\bullet}/\mathscr{P}_{\bullet}\to 0$ as \textbf{pure} whenever it corresponds to a pure subcomplex in this sense.
\end{definition}

The following lemma extends Lemma~\ref{lemma-pure cardinality}. Assume here that $\lambda$ is a regular cardinal for which  Ch$(\Qco(X))$ is locally $\lambda$-presentable.

\begin{lemma}\label{pure.subcomp.sheaf}
There exists a regular cardinal $\gamma > \lambda$ so that given any complex of sheaves $\mathscr{X}_{\bullet}$ and $\mathscr{S}_{\bullet}\subseteq \mathscr{X}_{\bullet}$, with $\mathscr{S}_{\bullet}$ a $\gamma$-presentable complex, there exists a pure subcomplex $\mathscr{P}_{\bullet}\subseteq \mathscr{X}_{\bullet}$ with $\mathscr{S}_{\bullet}\subseteq \mathscr{P}_{\bullet}$ and $\mathscr{P}_{\bullet}$ also $\gamma$-presentable.
\end{lemma}

\begin{proof}
By~\cite[Theorem 2.33]{adamek-rosicky} there exist (arbitrarily large) regular cardinals $\gamma > \lambda$ with the stated property but instead promising that $\mathscr{P}_{\bullet}\subseteq \mathscr{X}_{\bullet}$ be a $\lambda$-pure subcomplex rather than pure in the sense of our Definition~\ref{def-pur}. But we will argue that it is pure in this sense too. Indeed, by~\cite[Proposition 2.30]{adamek-rosicky} the $\lambda$-pure morphism $\mathscr{P}_{\bullet}\subseteq \mathscr{X}_{\bullet}$ is a $\lambda$-directed colimit of split monomorphisms. In particular, for every open affine $U\subseteq X$, the functor $\Gamma(U,-)$ commutes with direct limits, and so we obtain a $\lambda$-directed colimit of split monomorphisms $\mathscr{P}_{\bullet}(U)\subseteq \mathscr{X}_{\bullet}(U)$ in Ch($\mathscr{R}(U)$). Since every $\lambda$-directed system is an $\omega$-directed system, it follows that $\mathscr{P}_{\bullet}(U)\subseteq \mathscr{X}_{\bullet}(U)$ is an $\omega$-directed colimit (i.e. a direct limit) of split monomorphisms in Ch($\mathscr{R}(U)$), whence a pure subcomplex in Ch($\mathscr{R}(U)$) in the sense of Definition~\ref{def-purity}.
\end{proof}

By using the previous lemma we get the corresponding version of Proposition~\ref{prop-transfinite} for chain complexes of sheaves.

\begin{proposition}\label{prop-transfinite-extens-sheaf}
Let $\mathcal A$ be a class of chain complexes of sheaves. Suppose $\mathcal A$ is closed under taking pure subcomplexes and pure quotients. Then there is a regular cardinal $\gamma$ such that every chain complex in $\mathcal A$ is a transfinite extension of complexes in $\mathcal A$ which are $\gamma$-presentable.
\end{proposition}

\begin{proof}
All the work has been previously done in the proof of Proposition 3.4. We only have to replace Lemma 3.3 by Lemma \ref{pure.subcomp.sheaf}  in the proof of Proposition 3.4 to get the desired transfinite extension of a complex $\mathscr{A}_{\bullet}\in \mathcal A$. Let us see why. Let $\gamma$ be some regular cardinal as in Lemma~\ref{pure.subcomp.sheaf} and let $\mathscr{A}_{\bullet} \in \mathcal{A}$. Again if $\mathscr{A}_{\bullet}$ is $\gamma$-presentable we are done. Otherwise, we start by applying Lemma \ref{pure.subcomp.sheaf} to find a nonzero pure subcomplex $\mathscr{A}^0_{\bullet}$ with $\mathscr{A}^0_{\bullet}$ $\gamma$-presentable. Then $\mathscr{A}^0_{\bullet}$ and $\mathscr{A}_{\bullet}/\mathscr{A}^0_{\bullet}$ are each complexes in $\mathcal A$ by our assumption on $\mathcal A$. So we again apply Lemma \ref{pure.subcomp.sheaf} to $\mathscr{A}_{\bullet}/\mathscr{A}^0_{\bullet}$ to get a nonzero $\gamma$-presentable pure subcomplex $\mathscr{A}^1_{\bullet}/\mathscr{A}^0_{\bullet}\subset \mathscr{A}_{\bullet}/\mathscr{A}^0_{\bullet} $. The proof of Proposition 3.4 shows that indeed $\mathscr{A}^1_{\bullet}(U)\subset \mathscr{A}_{\bullet}(U)$ is pure in Ch($\mathscr{R}(U)$). Hence $\mathscr{A}^1_{\bullet}\subset \mathscr{A}_{\bullet}$ is pure (and so $\mathscr{A}_{\bullet}/\mathscr{A}^1_{\bullet}$ is a pure quotient). By repeating this procedure we may construct a strictly increasing chain, $0\neq \mathscr{A}^0_{\bullet}\subsetneq  \mathscr{A}^1_{\bullet}\subsetneq \ldots \subsetneq \mathscr{A}^n_{\bullet}\subsetneq  \dots$, where each $\mathscr{A}_{\bullet}^n$ is a pure subcomplex of $\mathscr{A}_{\bullet}$ and each quotient $\mathscr{A}_{\bullet}^{n+1}/\mathscr{A}_{\bullet}^n$ is $\gamma$-presentable and belongs to $\mathcal A$. We set $\mathscr{A}_{\bullet}^{\omega}=\cup_{n<\omega} \mathscr{A}_{\bullet}^n$. Now it is easy to see that our notion of purity in Ch($\Qco(X)$) is closed under direct unions (again because the notion of purity in Ch($\mathscr{R}(U)$) is such). So again $\mathscr{A}_{\bullet}^{\omega}$ and $\mathscr{A}_{\bullet}/\mathscr{A}_{\bullet}^{\omega}$ are also each in $\mathcal A$. So we can follow the process by transfinite induction, setting $\mathscr{A}^{\mu}_{\bullet}=\cup_{\alpha<\mu} \mathscr{A}^{\alpha}_{\bullet}$ whenever $\mu$ is a limit ordinal. The process must stop at some step, so we get the desired result.
\end{proof}

\subsection{Acyclicity for complexes of flat sheaves}
We now extend the notion of $A$-acyclic complexes from Section~\ref{sec-acyclicity-modules}, from modules to sheaves.  The following is the analog to Definition~\ref{def-A-acyclicity}.

\begin{definition}\label{def-A-acyclicity-sheaves}
Let $\mathscr{F}_{\bullet}$ be a chain complex of flat sheaves.
Suppose that for each open affine $U \subseteq X$, we are given an $\mathscr{R}(U)$-module $A_U$. We will then say that $\mathscr{F}_{\bullet}$ is  \textbf{$\boldsymbol{A_X}$-acyclic} if $\mathscr{F}_{\bullet}(U)$ is an $A_U$-acyclic complex for every open affine $U$ of $X$; that is, $A_U\otimes_{\mathscr{R}(U)}\mathscr{F}_{\bullet}(U)$ is exact for every open affine $U\subseteq X$.
\end{definition}

We now show that Theorem~\ref{them-deconstructibility of tensor acyclic complexes} extends to complexes of sheaves.

\begin{theorem}\label{them-A-cot-pairs-sheaves}
Let $X$ be a scheme with structure sheaf $\mathscr{R}$.
Suppose that for each open affine $U \subseteq X$, we are given an $\mathscr{R}(U)$-module $A_U$.
Let $\barAF$ denote the class of all $A_X$-acyclic complexes of flat sheaves.
Then $\barAF$ is a covering class. Moreover, if $\Qco(X)$ has a flat generator (for instance if $X$ is quasi-compact and semi-separated), then every $\barAF$-cover is an epimorphism and $(\barAF,\barAF^{\perp})$ is a complete hereditary cotorsion pair cogenerated by a set.
\end{theorem}

\begin{proof}
Let us first see that the class $\barAF$ is closed under taking pure subcomplexes and pure quotients. Let $0\to \mathscr{P}_{\bullet}\to \mathscr{F}_{\bullet}\to \mathscr{G}_{\bullet}\to 0$ be a pure exact sequence in Ch$(\Qco(X))$ with $\mathscr{F}_{\bullet} \in {}_{A_{\! X}\! }{\widetilde{\mathcal F}}$. Then, for every open affine $U$ we have the pure exact sequence $0\to \mathscr{P}_{\bullet}(U)\to \mathscr{F}_{\bullet}(U)\to \mathscr{G}_{\bullet}(U)\to 0$ in Ch($\mathscr{R}(U)$). The proof of Theorem~\ref{them-deconstructibility of tensor acyclic complexes} tells us then that $\mathscr{P}_{\bullet}(U)$ and $\mathscr{G}_{\bullet}(U)$ are both $A_U$-acyclic. Hence, $\mathscr{P}_{\bullet}$ and $\mathscr{G}_{\bullet}$ are in $\barAF$, proving that $\barAF$ is closed under pure subobjects and pure quotients. We infer from Proposition~\ref{prop-transfinite-extens-sheaf} that there is a set $\mathcal S$ such that every complex in $\barAF$ is a transfinite extension of $\mathcal S$. The class $\barAF$ is clearly closed under direct limits and extensions (again because the question is easily checked at the level of sections on each open affine $U$). Therefore $\barAF$ is a covering class. Moreover, if $\Qco(X)$ possesses a flat generator $\mathscr{G}$, then the disks $D^n(\mathscr{G})$ are clearly generators for Ch$(\Qco(X))$ that lie in $\barAF$. Hence it follows that $(\barAF,\barAF^{\perp})$ is a complete (indeed perfect) cotorsion pair in Ch$(\Qco(X))$. The fact that $\barAF$ is closed under pure subcomplexes gives us immediately that the complete cotorsion pair $(\barAF, \barAF^{\perp})$ is an hereditary one.
\end{proof}

\subsection{AC-acyclic complexes of flat sheaves}
We now also extend, from modules to sheaves, the notion of (exact) AC-acyclic complexes of flats. So the following is based on Definition~\ref{def-acyclicity}.

\begin{definition}\label{def-AC-sheaves}
Let $\mathscr{F}_{\bullet}$ be a chain complex of flat sheaves.
We say that $\mathscr{F}_{\bullet}$ is  \textbf{AC-acyclic} if $\mathscr{F}_{\bullet}(U)$ is an AC-acyclic complex for every open affine $U$ of $X$; that is, $A_U \otimes_{\mathscr{R}(U)}\mathscr{F}_{\bullet}(U)$ is exact for every absolutely clean $\mathscr{R}(U)$-module $A_U$.  If $\mathscr{F}_{\bullet}$ is  itself exact we say it is \textbf{exact AC-acyclic}. Note that this is equivalent to saying that $\mathscr{F}_{\bullet}(U)$ is an exact AC-acyclic complex for every open affine $U$ of $X$.
\end{definition}

Using Setup~\ref{setup-choice of A} we can, for each open affine $U \subseteq X$, find an $\mathscr{R}(U)$-module $A_U$ such that $\mathscr{F}_{\bullet}$ is $A_X$-acyclic if an only if it is (exact) AC-acyclic. So we get the following result.

\begin{corollary}\label{cor-AC-cot-pairs-sheaves}
Let $X$ be a scheme with structure sheaf $\mathscr{R}$.
Let $\barAF$ denote the class of all AC-acyclic complexes of flat sheaves.
Then $\barAF$ is a covering class. Moreover, if $X$ is quasi-compact and semi-separated, then every $\barAF$-cover is an epimorphism and $(\barAF,\barAF^{\perp})$ is a complete hereditary cotorsion pair cogenerated by a set. The same results hold if we replace $\barAF$ with the class of all exact AC-acyclic complexes of flat sheaves.
\end{corollary}

To be a good notion for complexes of flat sheaves, one wants the notion of (exact) AC-acyclicity to be a \emph{Zariski-local} property; that is, a property that can be checked by using any open affine cover of $X$. We refer the reader to~\cite{chris-estrada-iacob} for a more detailed discussion of Zariski-local properties. Unfortunately, we do not know whether or not the notion of AC-acyclic complexes of flats is always a Zariski-local property. However, from the work in~\cite{chris-estrada-iacob}, we can now easily conclude it is indeed a Zariski-local property whenever the underlying scheme $X$ is locally coherent.

\begin{definition}\label{def-coherent-scheme}
A scheme $X$ with structure sheaf $\mathscr{R}$ is called \textbf{locally coherent} if $\mathscr{R}(U)$ is a coherent ring for every open affine $U \subseteq X$. This notion is Zariski-local by ~\cite[Proposition~4.9]{chris-estrada-iacob}, so it can be tested using any affine open covering of $X$. We say $X$ is \textbf{coherent} if it is locally coherent and quasi-compact.
\end{definition}

\begin{corollary}\label{cor-zariski-local}
Let $X$ be a locally coherent scheme  with structure sheaf $\mathscr{R}$. Then $\mathscr{F}_{\bullet}$ is an (exact) AC-acyclic complex of flat sheaves if and only if there exists an affine open covering $\mathcal U$ of $X$ such that $\mathscr{F}_{\bullet}(U)$ is an (exact) AC-acyclic complex of flat $\mathscr{R}(U)$-modules, for every $U \in \mathcal U$. In other words, the notion of (exact) AC-acyclicity is a Zariski-local property.
\end{corollary}

\begin{proof}
Since for coherent rings, absolutely clean modules coincide with absolutely pure modules, it follows from Lemma~\ref{lemma-AP-acyclic} that $\mathscr{F}_{\bullet}(U)$ is AC-acyclic if and only if $I \otimes_{\mathscr{R}(U)} \mathscr{F}_{\bullet}(U)$ is exact for each open affine $U \subseteq X$ and injective module $I \in \mathscr{R}(U)$-Mod. But this is directly shown to be a Zariski-local property in~\cite[Proposition~3.7]{chris-estrada-iacob}. We note that the authors' proof of this works in the same way, regardless of whether or not $\mathscr{F}$ is exact from the start.  That is, both AC-acyclicity and exact AC-acyclicity are Zariski-local.
\end{proof}

\subsection{AC-acyclicity over a semi-separated and locally coherent scheme}
Our main goal now is to prove the following characterization of (exact) AC-acyclic complexes of flats over a semi-separated and locally coherent  scheme. A  (quasi-coherent) sheaf $\mathscr{A}$ is called \textbf{locally absolutely pure}  if $\mathscr{A}(U)$ is an absolutely pure $\mathscr{R}(U)$-module for each open affine $U \subseteq X$. They were introduced in~\cite[Section~5]{EEO-pure} and shown there to have very satisfying properties over locally coherent schemes.

\begin{theorem}\label{them-AC-characterization}
Let $X$ be semi-separated and locally coherent. Then an (exact) chain complex $\mathscr{F}_{\bullet}$ of flat sheaves is (exact) AC-acyclic if and only if $\mathscr{A} \otimes \mathscr{F}_{\bullet}$ is exact for all locally absolutely pure sheaves $\mathscr{A}$.
\end{theorem}

We break the proof of the theorem down into a series of lemmas.

\begin{lemma}\label{lemma-1}
Let $R$ be a commutative coherent ring. The affine scheme $X = \textnormal{Spec}\,R$, with structure sheaf $\mathscr{O}_X$, is a locally coherent scheme. That is, $\mathscr{O}_X(U)$ is a coherent ring for each affine open $U$ of $X$.
\end{lemma}

\begin{proof}
Local coherence is a Zariski-local property by ~\cite[Proposition~4.9]{chris-estrada-iacob}.
\end{proof}

\begin{lemma}\label{lemma-2}
With the same setup as Lemma~\ref{lemma-1}, let $A$ be an absolutely pure $R$-module. Then the associated sheaf $\widetilde{A}$ on $X$ is locally absolutely pure. That is, $\widetilde{A}(U)$ is an absolutely pure $\mathscr{O}_X(U)$-module for each affine open $U$ of $X$.
\end{lemma}

\begin{proof}
It is a basic fact that $\Gamma(X,\widetilde{A}) = A$. That is, $\widetilde{A}(X) = A$ is an absolutely pure $R$-module. So this is immediate from Lemma~\ref{lemma-1} and~\cite[Prop.~5.7]{EEO-pure}, which shows that local absolute purity is a Zariski-local property whenever the scheme is locally coherent.
\end{proof}

\begin{lemma}\label{lemma-3}
Let $f : R \xrightarrow{} S$ be a flat ring homomorphism. If $A$ is an absolutely pure $S$-module, then it is also an absolutely pure $R$-module upon restricting scalars.
\end{lemma}

\begin{proof}
We have the standard tensor-Hom adjoint relationship $\Hom_S(M \otimes_R S , N) = \Hom_R(M, \Hom_S(S,N))$, where $\Hom_S(S,N)$ inherits its $R$-module structure via $f$. So in this case, $\Hom_S(S,N) = {}_RN$ coincides with restriction of scalars along $f$. Hence the adjunction is $\Hom_S(M \otimes_R S , N) = \Hom_R(M, {}_RN)$. We note that $- \otimes_R S$ is both exact and preserves projective modules. Similarly ${}_R(-) = \Hom_S(S,-)$ is exact and preserves injectives, since $\Hom_R(-, {}_RN) = \Hom_S(- \otimes_R S , N)$. In any case, it becomes clear that the adjunction extends to $\Ext^n_S(M \otimes_R S , N) = \Ext^n_R(M, {}_RN)$ for each $n$. It is easy to see that if $M$ is a finitely presented $R$-module, then $M \otimes_R S$ is a finitely presented $S$-module. Hence taking $N = A$ to be an absolutely pure (in other words FP-injective) $S$-module, we see that $\Ext^1_R(M, {}_RA) = \Ext^1_S(M \otimes_R S , N) = 0$ for all finitely presented $R$-modules $M$. So ${}_RA$ is an FP-injective $R$-module, in other words, an absolutely pure $R$-module.
\end{proof}

\begin{lemma}\label{lemma-4}
Let $X$ be a semi-separated and locally coherent scheme with structure sheaf $\mathscr{R}$. Let $\{U_i\}$ be a semi-separating affine basis. Then for any inclusion $j : U_i \hookrightarrow X$, if $A$ is an absolutely pure $\mathscr{R}(U_i)$-module, then the direct image sheaf $j_*(\widetilde{A})$ is a locally absolutely pure (quasi-coherent) $\mathscr{R}$-module.
\end{lemma}

\begin{proof}
Because $X$ is semi-separated, the direct image functor $j_*(\widetilde{A})$ preserves quasi-coherent. So the goal is to show that this is a locally absolutely pure sheaf. Since $X$ is locally coherent it is enough, by~\cite[Prop.~5.7]{EEO-pure},  to show that $[j_*(\widetilde{A})]( U_k)$ is an absolutely pure $\mathscr{R}(U_k)$-module for each $U_k \in \{U_i\}$. By the definition of $j_*$ we have $[j_*(\widetilde{A})](U_k) = \widetilde{A}(U_i  \cap U_k)$. The assumption that $\{U_i\}$ is a semi-separating affine basis means that $U_i \cap U_k$ is again affine open. So $\widetilde{A}(U_i  \cap U_k)$ is certainly absolutely pure as an $\mathscr{R}(U_i \cap U_k)$-module, by Lemma~\ref{lemma-2}. To complete the proof, we need it to be absolutely pure as an $\mathscr{R}(U_k)$-module.  However, the inclusion $U_i \cap U_k  \hookrightarrow U_k$ corresponds to a flat ring homomorphism $f : \mathscr{R}(U_k) \xrightarrow{} \mathscr{R}(U_i \cap U_k)$, and viewing $\widetilde{A}(U_i  \cap U_k)$ as a $\mathscr{R}(U_k)$-module corresponds to restriction of scalars along the homomorphism, by~\cite[Prop.~5.2(d)]{hartshorne}. So applying Lemma~\ref{lemma-3} completes the proof.
\end{proof}

We are now ready to prove the theorem.

\begin{proof}[Proof of Theorem~\ref{them-AC-characterization}]
Suppose  $\mathscr{F}_{\bullet}$ is (exact) AC-acyclic in the sense of Definition~\ref{def-AC-sheaves}. This means it satisfies $A_U \otimes  \mathscr{F}_{\bullet}(U)$ is exact for all open affine $U \subseteq X$ and every absolutely pure $\mathscr{R}(U)$-module $A_U$. Let $\mathscr{A}$ be a locally absolutely pure sheaf. For each open affine $U$ of $X$, we do have $(\mathscr{A} \otimes  \mathscr{F}_{\bullet})(U) = \mathscr{A}(U) \otimes  \mathscr{F}_{\bullet}(U)$. Hence $(\mathscr{A} \otimes  \mathscr{F}_{\bullet})(U)$ is exact for all open affine $U$, which implies the complex $\mathscr{A} \otimes  \mathscr{F}_{\bullet}$ is exact.

For the converse, let $\{U_i\}$ be a semi-separating affine basis. Assume $\mathscr{A} \otimes \mathscr{F}_{\bullet}$ is exact for all locally absolutely pure sheaves $\mathscr{A}$. Since we showed in Corollary~\ref{cor-zariski-local} that (exact) AC-acyclicity is a Zariski-local property, we only need to show that $\mathscr{F}_{\bullet}(U_i)$ is an (exact) AC-acyclic complex of flat $\mathscr{R}(U_i)$-modules for each $U_i$. Certainly $\mathscr{F}_{\bullet}(U_i)$ is an (exact) complex of flat sheaves, so let $A$ be an absolutely pure $\mathscr{R}(U_i)$-module. To complete the proof we need to show that $A \otimes \mathscr{F}_{\bullet}(U_i)$ is exact. But letting $j : U_i \hookrightarrow X$ be the inclusion, we have from Lemma~\ref{lemma-4} that the direct image sheaf $j_*(\widetilde{A})$ is a locally absolutely pure (quasi-coherent) $\mathscr{R}$-module. Hence $j_*(\mathscr{A}) \otimes \mathscr{F}_{\bullet}$ is exact. In particular, we get that $(j_*(\widetilde{A}) \otimes \mathscr{F}_{\bullet})(U_i)$ is an exact complex of $\mathscr{R}(U_i)$-modules. But since $U_i$ is an affine open subset we have $(j_*(\widetilde{A}) \otimes \mathscr{F}_{\bullet})(U_i) = [j_*(\widetilde{A}) ](U_i) \otimes \mathscr{F}_{\bullet}(U_i)$. But by the very definition of $j_*$, we easily see $[j_*(\widetilde{A}) ](U_i) = A$. Thus $A \otimes \mathscr{F}_{\bullet}(U_i)$ is exact as desired.
\end{proof}

\section{Flat model structures and localization sequences}\label{sec-AC-acyclic flat model on sheaf}

At this point it is quite easy to reach our goal of extending the results of Section~\ref{sec-AC-acyclic flat model on ch} to complexes of quasi-coherent sheaves. After doing this we will use the resulting model structures to obtain localization sequences. We start with the following extension of Theorem~\ref{thm-how to create flat on chain}. We interpret it as providing a way to extend homotopy categories of $A$-acyclic complexes of projectives to non-affine schemes. Recall our convention from the last section: All sheaves are quasi-coherent even if not explicitly stated.

\begin{theorem}\label{thm-how to create flat on sheaf}
Let $X$ be a quasi-compact and semi-separated scheme with structure sheaf $\mathscr{R}$.
Suppose that for each open affine $U \subseteq X$, we are given an $\mathscr{R}(U)$-module $A_U$.
Let $\barAF$ denote the class of all $A_X$-acyclic complexes of flat sheaves (see Definition~\ref{def-A-acyclicity-sheaves}). Then there is a cofibrantly generated abelian model structure, the \textbf{$\boldsymbol{A_X}$-acyclic flat model structure}, $\class{M} = ({}_A\tilclass{F}, \class{W}, \dgclass{C})$ on $\textnormal{Ch}(\Qco(X))$, for which ${}_A\tilclass{F} \cap \class{W}$ equals the class  $\tilclass{F}$ of flat chain complexes. If $X = \textnormal{Spec}\,R$ is an affine scheme, then this model structure coincides with the one in Theorem~\ref{thm-how to create flat on chain}. Therefore, in the affine case, we have $\textnormal{Ho}(\class{M}) \cong K_{\text{A-ac}}(Proj)$, the chain homotopy category of all $A$-acyclic complexes of projectives.
\end{theorem}

\begin{proof}
Follow the first paragraph of the proof of Theorem~\ref{thm-how to create flat on chain}. It is well known that there is a complete hereditary cotorsion pair $(\tilclass{F}, \dgclass{C})$ on complexes of sheaves, where $\tilclass{F}$ is the class of flat complexes. For example, see~\cite[Theorem~6.7]{gillespie-quasi-coherent}. For the other cotorsion pair, $(\barAF,\barAF^{\perp})$, we appeal to Theorem~\ref{them-A-cot-pairs-sheaves}. But everything else works the same way.
\end{proof}

\begin{example}\label{example-mock}
Suppose that for each open affine $U \subseteq X$, we take $A_U = 0$ for our choice of each $\mathscr{R}(U)$-module. Then all complexes $\mathscr{F}_{\bullet}$ of flat sheaves  are $A_X$-acyclic. In this case, the $A_X$-acyclic model structure has homotopy category equivalent to $K_m(Proj X)$, Murfet's \emph{mock homotopy category of projectives} from~\cite{murfet-thesis}. Similarly, if for each open affine $U \subseteq X$, we take $A_U = \mathscr{R}(U)$, then a complex $\mathscr{F}_{\bullet}$ of flat sheaves is $A_X$-acyclic if and only if it is exact. In this case, the $A_X$-acyclic model structure has homotopy category equivalent to
 $K_{m,ac}(Proj X)$, Murfet's \emph{mock projective stable derived category} from~\cite{murfet-thesis}. For more details on these model structures we refer the reader to~\cite{gillespie-mock projectives}.
\end{example}

In the same way that we obtained Corollary~\ref{cor-AC-cot-pairs-sheaves}, we may use the idea in Setup~\ref{setup-choice of A} to find, for each open affine $U \subseteq X$, an $\mathscr{R}(U)$-module $A_U$ such that $\mathscr{F}_{\bullet}$ is $A_X$-acyclic if an only if it is (exact) AC-acyclic. Applying Theorem~\ref{thm-how to create flat on sheaf} gives us the following.

\begin{corollary}\label{cor-exact-AC-flat-model-sheaf}
Let $X$ be a quasi-compact and semi-separated scheme.
Let $\barAF$ denote the class of all (exact) AC-acyclic complexes of flat sheaves (see Definition~\ref{def-AC-sheaves}). Then there is a cofibrantly generated abelian model structure, the \textbf{(exact) AC-acyclic flat model structure}, $\class{M} = (\barAF, \class{W}, \dgclass{C})$ on $\textnormal{Ch}(\Qco(X))$, for which $\barAF \cap \class{W}$ equals the class  $\tilclass{F}$ of flat chain complexes. If $X = \textnormal{Spec}\,R$ is an affine scheme, then this model structure coincides with the one in Theorem~\ref{cor-exact-AC-flat-model}. Therefore, in the affine case, we have $\textnormal{Ho}(\class{M}) \cong K_{\text{fir}}(Proj)$, the chain homotopy category of all (exact) firmly acyclic complexes of projectives.
\end{corollary}

As we already pointed out, we do not know whether or not the notion of AC-acyclicity is Zariski-local in full generality. But for $X$ semi-separated and locally coherent we have Theorem~\ref{them-zariski-local}. Now recall that by a \emph{coherent scheme} we mean one that is quasi-compact and locally coherent. Thus Theorem~\ref{them-coherent-scheme-model} of the introduction follows as a special case of the above Corollary~\ref{cor-exact-AC-flat-model-sheaf}.

\subsection{Adjoints of homotopy categories}
We end by pointing out the existence of certain adjoints and localization sequences involving the homotopy categories we have constructed. Each result is an easy corollary of a known general theorem.  First, by combining Corollary~\ref{cor-AC-cot-pairs-sheaves} with~\cite[Prop.~1.4]{neeman-adjoints}, we obtain the following.

\begin{corollary}\label{cor-adjoints}
Let $X$ be any scheme and $K(Flat X)$ the chain homotopy category of all flat sheaves. Denote by $K_{\scriptscriptstyle{AC}}(Flat X)$, the full subcategory of either the AC-acyclic or exact AC-acyclic complexes of flats, as in Definition~\ref{def-AC-sheaves}. In either case, the inclusion $K_{\scriptscriptstyle{AC}}(Flat X) \xrightarrow{} K(Flat X)$ has a right adjoint.
\end{corollary}

We note that the proof of the above only relies on the existence of (exact) AC-acyclic flat precovers. In particular it does not require the existence of a set of flat generators. However, if we are willing to assume we have such a set of generators then we get an explanation for the adjunction by way of cotorsion pairs and abelian model categories. In fact, in the spirit of Becker's left localization technique from~\cite[Prop.~1.4.6]{becker}, we can at once deduce a localization sequence with all three categories appearing as homotopy categories of hereditary abelian model structures.
To state the result, we let $\class{M}_1 = (\dwclass{F},\class{W}_1,\dgclass{C})$ denote the abelian model structure of~\cite[Corollary~4.1]{gillespie-mock projectives}. Here $\dwclass{F}$ denotes the class of all complexes of flat sheaves,  $\dwclass{F} \cap \class{W}_1 = \tilclass{F}$ is the class of categorically flat sheaves, and $\dgclass{C} = \rightperp{\tilclass{F}}$. As noted in Example~\ref{example-mock} above, this is a model structure for Murfet's mock homotopy category of projectives, which can also be thought of as $\class{D}(Flat X)$, the derived category of flat sheaves. We now let  $\class{M}_2 = (\barAF, \class{W}_2, \dgclass{C})$ denote either the AC-acyclic model structure or the exact AC-acyclic model structure from Corollary~\ref{cor-exact-AC-flat-model-sheaf}. Then in either case, the following holds by~\cite[dual of Prop.~3.2]{gillespie-mock projectives}.

\begin{corollary}\label{cor-localization}
Let $X$ be any quasi-compact and semi-separated scheme. Then there is a localization sequence
\[
\begin{tikzpicture}[node distance=3.5 cm, auto]
\node (A)  {$\textnormal{Ho}(\class{M}_2)$};
\node (B) [right of=A] {$\textnormal{Ho}(\class{M}_1)$};
\node (C) [right of=B] {$\textnormal{Ho}(\class{M}_2\backslash\class{M}_1)$};
%
%
\draw[->] (A.10) to node {$L(\textnormal{Id})$} (B.170);
\draw[<-] (A.350) to node [swap] {$R(\textnormal{Id})$} (B.190);
\draw[->] (B.10) to node {$L(\textnormal{Id})$} (C.173);
\draw[<-] (B.350) to node [swap] {$R(\textnormal{Id})$} (C.187);
\end{tikzpicture}
\]
where $L(\textnormal{Id})$ and $R(\textnormal{Id})$ are left and right derived identity functors on $\textnormal{Ch}(\Qco(X))$ and $\class{M}_2 \backslash \class{M}_1$ is the \emph{left localization of $\class{M}_1$ by $\class{M}_2$}. By its construction, it is the Hovey triple $\class{M}_2 \backslash \class{M}_1 = (\dwclass{F}, \class{V}, \class{W}_2 \cap \dgclass{C})$ where $\class{V}$ is a uniquely determined thick class as described in~\cite[dual of Prop.~3.1]{gillespie-hovey triples, gillespie-mock projectives}.
\end{corollary}

In particular, Corollary~\ref{cor-localization-sequence} of the Introduction is a special case of the above.


\providecommand{\bysame}{\leavevmode\hbox to3em{\hrulefill}\thinspace}
\providecommand{\MR}{\relax\ifhmode\unskip\space\fi MR }
\providecommand{\MRhref}[2]{%
  \href{http://www.ams.org/mathscinet-getitem?mr=#1}{#2}
}
\providecommand{\href}[2]{#2}

\end{document}